\documentclass[12pt]{amsart}
\usepackage{authblk}
\usepackage{stmaryrd}
\usepackage{amsmath}
\usepackage{amsfonts}
\usepackage{amssymb}
\usepackage{amsthm}
\usepackage{mathabx}
\usepackage[all,2cell]{xy}
\usepackage{color}
\usepackage{etoolbox}
\usepackage{tikz-cd}
\usepackage{relsize}
\usepackage{IEEEtrantools}
\usepackage{xifthen}
\UseTwocells
\newbool{hidedetails}
\booltrue{hidedetails}
\newtheorem{lemma}{Lemma}
\newtheorem{proposition}{Proposition}

\newtheorem{theorem}{Theorem}
\newtheorem*{thm}{Theorem}
\theoremstyle{definition}
\newtheorem{definition}{Definition}

\newtheorem{remark}{Remark}

\newcommand{\blue}{\color{blue}}

\newcommand{\hide}[1]{\ifbool{hidedetails}{}{{\blue #1}}}

\newcommand{\dramo}[3][]{{\ifthenelse{\isempty{#1}}{{\mathcal M}_{#2}^{#3}}
		{{\mathcal M}_{#2}^{#3}\lp #1\rp}}}

\newcommand{\Hom}[1][]{{{\rm Hom}_{#1}}}

\newcommand{\rhom}[1][]{{{\mathbb R}{\rm Hom}_{#1}}}
\newcommand{\qcoh}[2][]{{\ifthenelse{\isempty{#1}}
		{\mathfrak{QC}_{#2}}{{\mathfrak QC}_{#2}\lp #1\rp}}}					

\newcommand{\gect}[1]{{\Gamma_{#1}}}
\newcommand{\gramo}[2][]{{\ifthenelse{\isempty{#1}}{{\mathcal M}_{#2}}{{\mathcal M}_{#2}\lp #1\rp}}}

\newcommand{\posi}{{\mathbb N}}											                                           
														                                    
\newcommand{\hooklongrightarrow}{\lhook\joinrel\longrightarrow}			
\newcommand{\w}[1]{{\widetilde{#1}}}

\newcommand{\oh}[1]{{\widehat{#1}}}

\newcommand{\e}[1]{{\emph{#1}}}
\newcommand{\ee}[1]{{{\bf #1}}}
\newcommand{\ce}[1]{{{\blue #1}}}
\newcommand{\cei}[1]{{{#1}}}

\newenvironment{eqn}{\begin{IEEEeqnarray*}{rCl}}{\end{IEEEeqnarray*}}
\newcommand{\n}{{\IEEEyesnumber}}

\newcommand{\F}{{\mathbb F}}

\newcommand{\hocolim}[1][]{{\underset{#1}{\rm hocolim}}}
\newcommand{\Ho}[1]{{{\rm Ho}\lp #1\rp}}

\newcommand{\noneg}{{\mathbb Z_{\geq 0}}}
\newcommand{\nopos}{{\mathbb Z_{\leq 0}}}
\newcommand{\homdi}{{\delta}}

\newcommand{\hilp}{{\overline{h}}}
\newcommand{\quack}[1][]{{\lb\Dquot_{\sche,\hilp}/\Gl{d}{\mathbb C}\rb}}
\newcommand{\stacks}[1]{{{\rm St}\lp #1\rp}}
\newcommand{\dgt}{{\mathbb Y}}
\newcommand{\spr}[1]{{{\rm SPr}\lp #1\rp}}
\newcommand{\yosh}[1]{{\w{#1}}}						
\newcommand{\yo}[1]{{\widehat{#1}}}						
\newcommand{\lifto}[1]{{\w{#1}}}
\newcommand{\oset}[1]{{\lf 0,\ldots,#1\rf}}
\newcommand{\sreso}[1][]{{\ifthenelse{\isempty{#1}}{\w{\mathcal R}}{{\w{\mathcal R}}\lp #1\rp}}}
\newcommand{\dmana}[1][]{{\w{\Dmana}}}
\newcommand{\I}{{\mathfrak I}}
\newcommand{\nuco}[1]{{\mathbf #1}}
\newcommand{\catpo}{{\nuco{0}}}
\newcommand{\reso}[1][]{{\ifthenelse{\isempty{#1}}{\mathcal R}{{\mathcal R}\lp #1\rp}}}
\newcommand{\kerma}[1][]{{K_{#1}}}

\newcommand{\pset}[1]{{P\lp #1\rp}}

\newcommand{\maspa}[2][\bullet]{{M_{#2}^{#1}}}
\newcommand{\laspa}[2][\bullet]{{L^{#1}_{#2}}}

\newcommand{\dga}[1][\bullet]{{A^{#1}}}
\newcommand{\ddga}[2][\bullet]{{A^{#1}_{#2}}}
\newcommand{\ddgb}[2][\bullet]{{B^{#1}_{#2}}}
\newcommand{\ddgc}[2][\bullet]{{C^{#1}_{#2}}}
\newcommand{\dgal}[2][\bullet]{{A^{#1}_{#2}}}
\newcommand{\denbu}[2][\bullet]{{{\mathcal D}^{#1}_{#2}}}
\newcommand{\genbu}[2][\bullet]{{{\mathcal E}^{#1}_{#2}}}
\newcommand{\denbub}[2][\bullet]{{\overline{\mathcal D}^{#1}_{#2}}}
\newcommand{\genbub}[2][\bullet]{{\overline{\mathcal E}^{#1}_{#2}}}
\newcommand{\henbu}[2][\bullet]{{{\mathcal G}^{#1}_{#2}}}
\newcommand{\fral}[2]{{{\mathbf F}_{#1}\lp #2\rp}}

\newcommand{\ske}[2]{{S_{#1,#2}}}

\newcommand{\kae}[1][\bullet]{{\Omega^{#1}}}
\newcommand{\hogy}[1][*]{{{ H}^{#1}}}
\newcommand{\derca}[1][]{{{\mathbf{A}}}}						
\newcommand{\dercac}[1][]{{\mathbf{DAlg}^{\rm c}}}
\newcommand{\we}[1][]{{{\mathcal W}_{#1}}}
\newcommand{\fib}[1][]{{{\mathcal F}_{#1}}}

\newcommand{\loca}[2][]{{{\mathcal L}_{#1}\lp #2\rp}}

\newcommand{\cocy}[1][]{{{\mathfrak C}_{#1}}}

\newcommand{\fibra}[2]{{{#1}\downarrow{#2}}}
\newcommand{\sderca}[1][]{{{{\mathbb{A}}}}}	
\newcommand{\fderca}[1][]{{\widetilde{\sderca[]}}}
\newcommand{\dgs}[1][]{{{\mathbb X}^{#1}}}

\newcommand{\hosh}[1][]{{{\mathcal H}^{#1}}}
\newcommand{\hofu}[1][]{{{\rm H}^{#1}}}
\newcommand{\ho}[1][]{{{\rm H}^{#1}}}
\newcommand{\Cs}[1][]{{{\rm Sch}}}
\newcommand{\cs}{{X}}
\newcommand{\streaf}[2][]{{{\mathcal O}^{#1}_{#2}}}
\newcommand{\smor}[1][\bullet]{{\phi^{#1}}}
\newcommand{\Dgs}[1][]{{{\mathfrak{S}}}}
\newcommand{\Dman}[1][]{{{{\mathfrak{M}}}}}
\newcommand{\Dgas}[1][]{{{\mathbb{AFM}}}}
\newcommand{\Dmana}[1][]{{{{\mathfrak A}}}}
\newcommand{\Gs}[1][]{{{\mathfrak{G}}}}
\newcommand{\pt}{{p}}

\newcommand{\lp}{\left(}
\newcommand{\rp}{\right)}
\newcommand{\lf}{\left\{}
\newcommand{\rf}{\right\}}
\newcommand{\lb}{\left[}
\newcommand{\rb}{\right]}

\newcommand{\sset}{{\rm SSet}}							
\newcommand{\mapis}[1][]{{{\rm Map}_{#1}}}					

\newcommand{\nerve}[1]{{{\mathcal N}\lp #1\rp}}					

\newcommand{\Gl}[2]{{{\rm GL}(#1,#2)}}						

\newcommand{\sche}{{\mathcal X}}							

\newcommand{\twish}[2][]{{\ifthenelse{\isempty{#1}}{\mathcal O_{#2}}{\mathcal O_{#2}(#1)}}}				
\newcommand{\hilpo}[1][]{{h_{#1}}}							

\newcommand{\simca}[1][]{{\mathbf{SAlg}}}						
\newcommand{\sima}[1]{{{#1}_\bullet}}						
\newcommand{\saff}[1][]{{\op{\simca}}}						
\newcommand{\op}[1]{{{\lp#1\rp}^{\rm op}}}						
\newcommand{\spec}[1]{{\mathbf{Spec}(#1)}}						
\newcommand{\spe}[1]{{{\rm Spec}(#1)}}	
\newcommand{\Dquot}{{\rm DQuot}}							

\newcommand{\Spec}{{\rm Spec}}

\tikzcdset{%
arrow style=tikz,%
diagrams={>={To[length=4.4pt,width=3.6pt,line width=0.4pt]}}%
}

\begin{document}

\title[Shifted symplectic structures on derived Quot-stacks I]{Shifted symplectic structures on derived Quot-stacks I\\ -- Differential graded manifolds --}
\maketitle

\author{Dennis Borisov, Ludmil Katzarkov, Artan Sheshmani}
{Dennis Borisov${}^{0}$ and Ludmil Katzarkov$^{3,4}$ and Artan Sheshmani${}^{1,2,3}$}
\address{${}^0$  Department of Mathematics and Statistics, University of Windsor, 401 Sunset Ave., Windsor Ontario, Canada}
\address{${}^1$  Center for Mathematical Sciences and Applications, Harvard University, Department of Mathematics, 20 Garden Street, Cambridge, MA, 02139}
\address{${}^2$ Center for quantum geometry of moduli spaces , Ny Munkegade 118
	Building 1530, DK-8000 Aarhus C, Denmark}
\address{${}^3$ National Research University Higher School of Economics, Russian Federation, Laboratory of Mirror Symmetry, NRU HSE, 6 Usacheva str.,Moscow, Russia, 119048}
\address{${}^4$ University of Miami, Department of Mathematics, 1365 Memorial Drive, Ungar 515, Coral Gables, FL 33146}
\date{\today}

\begin{abstract} A theory of dg schemes is developed so that it becomes a homotopy site, and the corresponding infinity category of stacks is equivalent to the infinity category of stacks, as constructed by Toën and Vezzosi, on the site of dg algebras whose cohomologies have finitely many generators in each degree. Stacks represented by dg schemes are shown to be derived schemes under this correspondence.
	
	\smallskip
	
	
	
\end{abstract}

\tableofcontents

\section*{Introduction}
Our goal is to develop explicit constructions of various derived stacks of coherent sheaves on Calabi--Yau manifolds and equip them with representatives of the corresponding shifted symplectic forms. Here is what we mean by the word \e{explicit}: the derived stacks should be quotients of actions of groups on some ``derived enhancements`` of usual projective schemes, these usual schemes and the group actions on them should be already known. This need for explicit constructions comes from the applications of shifted symplectic forms developed in \cite{BSY2}.

Demanding that our stacks are quotients of schemes means that we look at bounded families of coherent sheaves. Hence, following the usual custom, we fix a Hilbert polynomial and look at the corresponding semi-stable sheaves. Then the underlying classical scheme is the usual Quot-scheme, and the group action is known as well (e.g.\@ \cite{HuyLehn}).

\smallskip

It remains to construct the ``derived enhancements''. The common approach to derived geometry is that of \cite{HAG1}, \cite{HAG2}. We look at schemes over $\mathbb C$, so for us this approach means substituting affine schemes over $\mathbb C$  with simplicial or differential non-positively graded $\mathbb C$-algebras and then working with hypersheaves valued in simplicial sets with respect to the étale topology. This setting is also where the shifted symplectic structures come from, as defined in \cite{PTVV13}. 

A more explicit description of ``derived enhancement'' is possible in the approach of \cite{DerQuot}, where dg schemes are defined to be the usual schemes together with an extra structure given in terms of sheaves of differential non-positively graded algebras. Therefore our goal would be achieved, if we do the following:\begin{enumerate}
	\item \label{FirstStep} Develop a theory of dg schemes so that it becomes a homotopy site,\footnote{By a homotopy site we mean a category enriched in simplicial sets and equipped with an $S$-topology as defined in \cite{HAG1}.} and the corresponding infinity category of stacks is equivalent to the infinity category of stacks (as constructed in  \cite{HAG1}, \cite{HAG2}) on the site of differential non-positively graded $\mathbb C$-algebras whose cohomologies have finitely many generators in each degree. Stacks represented by dg schemes should be derived schemes under this correspondence.
	\item \label{SecondStep} Develop a procedure, that for each smooth projective scheme $\lp\sche,\twish[1]{\sche}\rp$ constructs dg Quot-schemes, together with the actions of the corresponding general linear groups. Under the correspondence from (\ref{FirstStep}) each such dg Quot-scheme should be equipped with a morphism into the stack of perfect complexes on $\sche$. This morphism should factor through the stacky quotient by the action of the general linear group, and the resulting morphism of stacks should be formally étale.
\end{enumerate}
In this paper we perform the first step. The second step is postponed to another paper for reasons to be discussed later in this Introduction.

\medskip

As far as the first step is concerned, there is nothing special about $\mathbb C$, so we do everything over a fixed ground field $\F$ of characteristic $0$. 

We need to have a manageable homotopy theory of dg schemes. In \cite{DerQuot} they are organized into a category.
We would like to be able to compute the mapping spaces in the simplicial localization of this category. This requires us to choose weak equivalences and to put some restrictions. \e{A dg manifold} is a dg scheme whose degree $0$ component is smooth and quasi-projective over $\F$, and the rest of the dg structure sheaf is locally almost free over the degree $0$-part. 

The notion of weak equivalences we use is stronger than the notion of quasi-isomorphisms introduced in \cite{DerQuot}. We put an additional condition for the morphisms to be almost affine (Definition \ref{NewDefinition}). The issue here is that a dg scheme can be quasi-isomorphic to the spectrum of an affine dg scheme, yet not being a spectrum itself (the ambient classical scheme might not be affine). Altogether we have the following

\begin{thm} (Thm.\@ \ref{FibrantObjects}) The category \cei{$\Dman[\F]$} of dg manifolds can be equipped with the structure of a category of fibrant objects, where fibrations are morphisms that are smooth in degree $0$ and locally almost free in negative degrees (Def.\@ \ref{DefFibrations}).\end{thm}

The main axioms for categories of fibrant objects are stability of (trivial) fibrations with respect to pullbacks, and existence of factorizations. They easily follow from our fibrations being smooth in degree $0$ and locally almost free in negative degrees.

The structure of a category of fibrant objects gives us a better way for computing mapping spaces, than just the hammock localization available in general (\cite{DK2} \S2). The mapping space between two dg manifolds is weakly equivalent to the nerve of the category of cocycles (i.e.\@ spans with the first leg being a trivial fibration) between the two objects. In other words the hammocks can be taken to be very short.

\smallskip

Given a differential non-positively graded $\F$-algebra $\dga$ we have the dg scheme $\cei{\spec{\dga}}$, which consists of the usual $\Spec\lp\dga[0]\rp$ and the sheaf of dg $\streaf{\Spec\lp\dga[0]\rp}$-algebras obtained by localizing $\dga$. If $\dga[0]$ is smooth over $\F$ and $\dga$ is almost free over $\dga[0]$, we have $\spec{\dga}\in\Dman[\F]$. This gives us the full subcategory $\Dmana[\F]\subset\Dman[\F]$ of \e{affine dg manifolds}. The opposite category $\sderca[\F]\cong\op{\Dmana[\F]}$ is a full subcategory of the category $\derca[\F]$ of all differential non-positively graded $\F$-algebras.

We use $\Dmana[\F]$ as the bridge between dg manifolds and affine derived schemes according to \cite{HAG1}, \cite{HAG2}. It will be obvious that $\Dmana[\F]$ inherits the structure of a category of fibrant objects from $\Dman[\F]$. Moreover, when restricted to $\Dmana[\F]$ our notion of weak equivalences coincides with that of quasi-isomorphisms. Using our quasi-projectivity assumption on dg manifolds, we establish the following

\begin{thm} (Thm.\@ \ref{ManifoldsInclusion}) The inclusion $\Dmana[\F]\hookrightarrow\Dman[\F]$ defines a homotopically full and faithful functor between the corresponding simplicial localizations (i.e.\@ the morphisms on mapping spaces are weak equivalences).\end{thm}

Next we need to compare mapping spaces in $\sderca[\F]$ and $\derca[\F]$. The most difficult result of this paper is the following

\begin{thm} (Thm.\@ \ref{AlgebrasInclusion}) The inclusion $\sderca[\F]\hookrightarrow\derca[\F]$ defines a homotopically full and faithful functor between the corresponding simplicial localizations.\end{thm}

Notice that $\op{\derca[\F]}$ is a model category, so it also has the notion of fibrations, but it is different from the one in $\Dmana[\F]\cong\op{\sderca}$. A cofibration in $\derca[\F]$ is a retract of an almost free morphism, or in terms of schemes a retract of a dg vector bundle. In $\Dmana[\F]$ a fibration needs to be almost free only in the negative degrees, while in degree $0$ it can be just smooth.

The key technical result needed to prove the last theorem is 

\begin{thm} (Thm.\@ \ref{FunctorialResolutions}) There is a functor $\sderca[\F]\rightarrow\sderca[\F]^{\Delta^{\rm op}}$ that defines simplicial resolutions with respect to the model structure on $\derca[\F]$.\end{thm}

The subtle point in this theorem is the functoriality of the resolution. The standard functorial factorization on $\derca[\F]$ uses free algebras generated by infinitely many elements, taking us out of $\sderca[\F]$. This subtlety makes the proof of this theorem very long, involving a lot of simplicial combinatorics.

\smallskip

Finally we address the comparison of the categories of pre-stacks and stacks valued in simplicial sets on the $3$ categories: $\op{\derca[\F]}$, $\Dmana\cong\op{\sderca[\F]}$ and $\Dman[\F]$. The fact that $\Dmana[\F]\hookrightarrow\op{\derca[\F]}$ and $\Dmana[\F]\hookrightarrow\Dman[\F]$ define homotopically full and faithful functors between simplicial localizations means that any $\dgs\in\Dmana[\F]$ defines the same pre-stack on $\Dmana[\F]$ whether computed within $\op{\derca[\F]}$ or within $\Dman[\F]$.

The essential image $\fderca[\F]\subset\derca[\F]$ of $\sderca[\F]$ consists of all $\dga$ s.t.\@ $\hogy\lp\dga\rp$ has finitely many generators in each degree. We observe that $\fderca[\F]\hookrightarrow\derca[\F]$ is a pseudo-model category as defined in \cite{HAG1}. Using the same étale topology as in \cite{ToVa} we obtain the categories of stacks $\stacks{\op{\fderca[\F]}}\simeq\stacks{\Dmana[\F]}$. This gives us half of the bridge we wanted. It remains to define stacks on $\Dman[\F]$ and connect them to stacks on $\Dmana[\F]$.

Étale coverings in $\op{\derca[\F]}$, and in $\Dmana[\F]$, are defined in terms of conditions on cohomology of dg algebras (we recall this definition in Section \ref{SectionStacks}). Essentially the same conditions give us a topology on $\Dman[\F]$. It is not surprising that we obtain the following

\begin{thm} (Thm.\@ \ref{QuillenEq}) The canonical adjunction between the categories of pre-stacks on $\Dmana[\F]$ and $\Dman[\F]$ defines a Quillen equivalence $\stacks{\Dmana[\F]}\simeq\stacks{\Dman[\F]}$.\footnote{Every quasi-isomorphism in $\Dman[\F]$ is locally (on the codomain) a weak equivalence, hence we believe that this theorem would remain valid, if we use quasi-isomorphisms instead of our weak equivalences.}\end{thm}

This theorem is not surprising because every dg manifold has an atlas of affine dg manifolds. To actually prove this theorem we need to show that such affine atlases define hypercovers in $\Dman[\F]$. We do this in Thm.\@ \ref{CechCovers}. Here we use a very nice feature of the mapping spaces in $\Dman[\F]$: let $\dgs\in\Dman[\F]$ and let $U\subseteq\dgs$ be an open dg subscheme, then for any $\dgs'\in\Dman[\F]$ the image of $\mapis[{\Dman[\F]}]\lp\dgs',U\rp\rightarrow\mapis[{\Dman[\F]}]\lp\dgs',\dgs\rp$ is weakly equivalent to a union of connected components (Lemma \ref{ConnectedComp}). This implies that the pre-stack $\yo{\dgs'\underset{\dgs}\times U}$ represented by $\dgs'\underset{\dgs}\times U$ is weakly equivalent to the homotopy pullback $\yo{\dgs'}\underset{\yo{\dgs}}\times^{h}\yo{U}$ (Prop.\@ \ref{TwoIntersections}). Then we obtain

\begin{thm} (Thm.\@ \ref{ManSche}) Let $\dgs\in\Dman[\F]$, and let $\yosh{\dgs}$ be the corresponding stack on $\Dmana[\F]$. Then $\yosh{\dgs}$ is a derived scheme as defined in \cite{To14}.\end{thm}

So, if we limit ourselves to dg algebras with cohomologies having finitely many generators in each degree, we find that the corresponding infinity category of stacks is equivalent to the infinity category of stacks on affine dg manifolds. When we include all dg manifolds, we still obtain an equivalent infinity category. 

Thus dg manifolds play the same role in derived geometry as quasi-projective schemes in the usual geometry: they define the same topos, but give a larger supply of explicitly definable objects, in particular the ones that are proper over $\F$.

\medskip

Now we turn to part (\ref{SecondStep}) of our program. The problem of constructing dg Quot-schemes was addressed in \cite{DerQuot}. The approach there was as follows: given a smooth projective scheme $\lp\sche,\twish[1]{\sche}\rp$ view the category of coherent sheaves on $\sche$ as the category of saturated graded $\gect{*}\lp\sche,\twish[1]{\sche}\rp$-modules and use the fact that the category of all $\gect{*}\lp\sche,\twish[1]{\sche}\rp$-modules has enough projective objects. Given a Hilbert polynomial $\hilpo{}$ realize the corresponding Quot-scheme as a closed subscheme of a finite product of Grassmannians and enhance this embedding to a structure of a dg manifold.

The main theorem in \cite{DerQuot} was a certain stabilization result: instead of taking the category of all $\gect{*}\lp\sche,\twish[1]{\sche}\rp$-modules one can work with a large enough truncation, i.e.\@ graded modules that become trivial in degrees of homogeneity larger than a fixed bound. Letting this bound go to $\infty$ produces in the limit the $\rhom$ complexes of the category of coherent sheaves on $\sche$. The main theorem of \cite{DerQuot} states that this limit is achieved at a finite bound. This would allow us to construct a dg manifold enhancement of the Quot-scheme within a finite product of Grassmannians, and this enhancement would have the correct tangent complexes.

While this paper was under review the authors of \cite{DerQuot} have alerted us to the fact that the main theorem of \cite{DerQuot} is wrong. Thus there is no dg Quot-scheme construction achieved in \cite{DerQuot}. There are still many dg manifolds constructed there, one for each bound. They have wrong tangent complexes, but the discrepancy appears in higher and higher cohomology degrees, as the bound tends to $\infty$. In addition the dg manifold at each bound projects canonically onto each one corresponding to a lower bound, i.e.\@ we have a projective system of dg manifolds. In a subsequent paper we will show that the limit of the corresponding projective system of stacks on $\Dman[\F]$ is a derived scheme and it is represented by a dg manifold. This will complete step (\ref{SecondStep}) of our program.

\subsection*{Acknowledgements} The first author would like to thank Tony Pantev and Dingxin Zhang for very helpful conversations. The third author would like to thank Dennis Gaitsgory, Tony Pantev, Fedor Bogomolov, Roman Bezrukovnikov and Vladimir Baranovsky, Sergey Arkhipov and Alina Marian for helpful discussions and commenting on the first versions of this article. Research of A.S. and D. B. was partially supported by generous Aarhus startup research grant of A. S., and partially by the NSF DMS-1607871, NSF DMS-1306313, the Simons 38558, and Laboratory of Mirror Symmetry NRU HSE, RF Government grant, ag. No 14.641.31.0001. A.S. would like to sincerely thank Professor Shing-Tung Yau for always being there to provide crucial help, kind advice and generous support. He would further like to sincerely thank the Center for Mathematical Sciences and Applications at Harvard University, the center for Quantum Geometry of Moduli Spaces at Aarhus University, and the Laboratory of Mirror Symmetry in Higher School of Economics, Russian federation for support. 

\medskip

\includegraphics[scale=0.2]{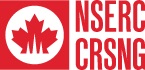} The first author acknowledges the support of the Natural Sciences and Engineering Research Council of Canada (NSERC), [RGPIN-2020-04845].

Cette recherche a été financée par le Conseil de recherches en sciences naturelles et en génie du Canada (CRSNG), [RGPIN-2020-04845].

\subsection*{Notation and conventions}

Unless stated otherwise all of our constructions happen over a fixed \ee{ground field} \ce{$\F$} of characteristic $0$. When dealing with $\F$-algebras we do not require $1\neq 0$, and hence allow $\{0\}$ as an $\F$-algebra. We do require morphisms of unital algebras to preserve the unit. We will not include $\F$ in the notation, but assume all of our algebras to be $\F$-algebras. Unless specified otherwise all tensor products are over $\F$. The category of differential non-positively graded $\F$-algebras will be denoted by \ce{$\derca[\F]$}.

\smallskip

This paper is concerned with comparing the theories of derived algebraic geometry over $\F$ developed by Ciocan-Fontanine -- Kapranov in \cite{DerQuot} and by Toën -- Vezzosi in \cite{HAG1}, \cite{HAG2}. Correspondingly we have tried to separate the terminology by calling the former constructions \ee{differential graded} and the latter \ee{derived}. E.g.\@ \ee{an affine derived scheme} is a derived scheme corresponding to a simplicial algebra as considered in \cite{To10}, while \ee{a dg scheme} is a dg enhancement of a usual scheme as defined in \cite{DerQuot}.

We use the \ee{cohomological notation}, i.e.\@ our differentials raise the degree by $1$. For a dg space \ce{$\dga$} we write \ce{$\dga[*]$} to denote \ee{the underlying graded space}, i.e.\@ with the differential forgotten. We say that $\ddga{}\rightarrow\ddgb{}$  is \ee{an almost free morphism of dg algebras}, if the underlying $\ddga[*]{}\rightarrow\ddgb[*]{}$ is free, i.e.\@ $\ddgb[*]{}$ is a free graded $\ddga[*]{}$-algebra.

For a dg algebra $\dga$ we use the term \ee{cotangent complex} to mean the homotopically correct cotangent complex of $\dga$, i.e.\@ the derived functor of the functor of usual \ee{Kähler differentials}. We denote the latter by \ce{$\kae[]$}.

When dealing with dg schemes we have two associated usual schemes: \ee{the underlying classical scheme}, i.e.\@ the closed subscheme given as cohomology in degree $0$, and \ee{the ambient classical scheme}, i.e.\@ the degree $0$ part of the dg scheme. We use the following notation: a dg scheme will be denoted by \ce{$\dgs$}, it consists of the ambient classical scheme \ce{$\lp\cs,\streaf{\cs}\rp$} and an extra sheaf of dg $\streaf{\cs}$-algebras \ce{$\streaf[\bullet]{\dgs}$}. The underlying classical scheme is \ce{$\lp\hofu[0]\lp\dgs\rp,\streaf{\hofu[0]\lp\dgs\rp}\rp$}.

\smallskip

We make use of a lot of \ee{homotopy colimits} (mostly of simplicial sets). There are two common approaches: as  derived functors of colimits and the Bousfield--Kan formula developed in \cite{BKan}. As shown in \cite{Gambino} these two approaches produce (weakly) equivalent results, so we use them interchangeably.

The category of ordered finite ordinals and weakly order preserving maps will be denoted by \ce{$\Delta$}. The \ee{category of simplicial sets} is denoted by \ce{$\sset$}. We will use categories enriched in $\sset$ calling them \ee{$\sset$-categories}. We will be especially interested in \ee{simplicial localizations} of usual categories obtained by taking the derived functor of inverting a set of morphisms. Given a category $\mathcal C$ with a chosen class of weak equivalences, its simplicial localization will be denoted by \ce{$\loca{\mathcal C}$}. \ee{The nerve of $\mathcal C$} will be denoted by \ce{$\nerve{\mathcal C}$}.

We will be working a lot with presheaves valued in $\sset$, and we will use hats to denote \ee{representable presheaves}.

\section{Dg manifolds}\label{DGMan}
\begin{definition} (\cite{DerQuot} \S2.2) \ee{A dg scheme} {$\dgs$} over $\F$ is given by a scheme {$\lp\cs,\streaf{\cs}\rp$} over $\F$ together with a sheaf $\streaf[\bullet]{\dgs}=\lp\streaf[*]{\dgs},\homdi\rp$ of differential $\nopos$-graded $\F$-algebras on $\cs$ and an isomorphism $\streaf{\cs}\cong\streaf[0]{\dgs}$, s.t.\@ $\forall k<0$ $\streaf[k]{\dgs}$ is a quasi-coherent sheaf of $\streaf[0]{\dgs}$-modules. We will call $\lp\cs,\streaf[0]{\cs}\rp$ \ee{the ambient classical scheme} for $\dgs$.
	
\ee{A morphism of dg schemes} $\dgs_1\rightarrow\dgs_2$ is given by a morphism of dg ringed spaces $\lp\Phi,\smor[\bullet]\rp\colon\lp\cs_1,\streaf[\bullet]{\dgs_1}\rp\rightarrow\lp\cs_2,\streaf[\bullet]{\dgs_2}\rp$
that restricts to a morphism of schemes $\lp\Phi,\smor[0]\rp\colon\lp\cs_1,\streaf[0]{\dgs_1}\rp\rightarrow\lp\cs_2,\streaf[0]{\dgs_2}\rp$. The category of dg schemes over $\F$ will be denoted by {$\Dgs[\F]$}.
	
\smallskip
	
The \ee{graded scheme $\lp\hofu[0]\lp\dgs\rp,\streaf[*]{\hofu\lp\dgs\rp}\rp$} associated to a dg scheme $\dgs$ consists of the closed subscheme $\iota\colon\lp\hofu[0]\lp\dgs\rp,\streaf[0]{\hofu\lp\dgs\rp}\rp\subseteq\lp\cs,\streaf{\cs}\rp$ defined by the sheaf of ideals $\homdi\lp\streaf[-1]{\dgs}\rp\subseteq\streaf[0]{\dgs}$, together with the sheaf $\streaf[*]{\hofu\lp\dgs\rp}:=\iota^*\lp\hosh[*]\lp\streaf[\bullet]{\dgs}\rp\rp$ of $\nopos$-graded $\F$-algebras where $\hosh[*]$ denotes the sheaf of cohomologies of $\streaf[\bullet]{\dgs}$. We will call $\lp\hofu[0]\lp\dgs\rp,\streaf[0]{\hofu\lp\dgs\rp}\rp$ \ee{the underlying classical scheme} for $\dgs$. 
	
Denoting by $\Gs[\F]\subset\Dgs[\F]$ the full subcategory consisting of {graded schemes} (i.e.\@ dg schemes with vanishing differential), we obtain a functor 
	\begin{eqn}\label{HomologyScheme}\hofu\colon\Dgs[\F]\longrightarrow\Gs[\F],\quad
		\dgs\longmapsto\lp\hofu[0]\lp\dgs\rp,\streaf[*]{\hofu\lp\dgs\rp}\rp.\n\end{eqn}
A morphism $\lp\Phi,\smor[\bullet]\rp$ in $\Dgs[\F]$ is \ee{a quasi-isomorphism}, if $\hofu\lp\Phi,\smor[\bullet]\rp$ is an isomorphism.\end{definition}
To be able to do homotopical algebra effectively we put some smoothness conditions on our dg schemes and require the ambient classical schemes to be quasi-projective. Our smoothness requirements are not invariant with respect to quasi-isomorphisms. They are similar to the conditions one puts on fibrant objects in homotopical algebra.

\begin{definition} A dg scheme $\dgs\in\Dgs[\F]$ is \ee{a dg manifold}, if {$\lp\cs,\streaf{\cs}\rp$} is a smooth quasi-projective scheme over $\F$, and for any affine open $U\subseteq \cs$ the sheaf $\streaf[*]{\dgs}\big\vert_U$ of $\nopos$-graded $\streaf{U}$-algebras is freely generated by a sequence of locally free coherent $\streaf{U}$-modules (one in each non-positive degree).\footnote{A choice of these generating locally free sheaves is \e{not} part of the structure of a dg manifold.} We denote by $\Dman[\F]\subset\Dgs[\F]$ the full subcategory consisting of dg manifolds.
	
\smallskip
	
Let $\derca[\F]$ be the category of differential $\nopos$-graded $\F$-algebras. The \ee{spectrum of $\dga\in\derca[\F]$} is the dg scheme {$\spec{\dga}$} having $\spe{{\dga[0]}}$ as the ambient classical scheme and $\streaf[\bullet]{\spec{\dga}}$ being given by localization of the dg $\dga[0]$-algebra $\dga$. Let $\sderca[\F]\subset\derca[\F]$ be the full subcategory consisting of $\dga$, s.t.\@ $\dga[0]$ is of finite type and smooth over $\F$, and the $\nopos$-graded $\dga[0]$-algebra $\dga[*]$ is freely generated by a sequence of finitely generated projective $\dga[0]$-modules (one in each non-positive degree). Restricting $\spec{-}$ to $\sderca[\F]$ we obtain the full subcategory $\Dmana[\F]\subset\Dman[\F]$ of \ee{affine dg  manifolds}.\end{definition}

\begin{remark} Since on affine schemes locally free coherent sheaves are projective objects, the property of being a dg manifold is local, i.e.\@ $\dgs\in\Dman[\F]$ if $\cs$ is quasi-projective over $\F$ and there is some open affine atlas $\cs=\underset{i\in I}\bigcup\,U_i$ s.t.\@ each $\lp U_i,\streaf[\bullet]{\dgs}\big\vert_{U_i}\rp$ is an affine dg manifold.
	
Clearly $\dgs\in\Dman[\F]$ is in $\Dmana[\F]$ if and only if $\cs$ is an affine scheme.\end{remark}
The advantage of $\Dman[\F]$ over $\Dgs[\F]$ is that we can view $\Dman[\F]$ as a category of fibrant objects with respect to a suitable notion of fibrations and weak equivalences.

\begin{definition}\label{DefFibrations} A $\lp\Phi,\smor[\bullet]\rp\colon\dgs_1\rightarrow\dgs_2$ in $\Dgs[\F]$ is a \ee{fibration} if:\begin{enumerate}
		\item The morphism $\lp\Phi,\smor[0]\rp\colon\lp\cs_1,\streaf{\cs_1}\rp\rightarrow\lp\cs_2,\streaf{\cs_2}\rp$ is smooth.
		\item For any $\pt\in\cs_1$ there are open affine subschemes $\pt\in U_1\subseteq\cs_1$ and $\Phi(\pt)\in U_2\subseteq\cs_2$ s.t.\@ denoting $\dgal{1}:=\Gamma\lp U_1, \streaf[\bullet]{\dgs_1}\rp$, $\dgal{2}:=\Gamma\lp U_2, \streaf[\bullet]{\dgs_2}\rp$ we have $\dgal{2}\underset{\dgal[0]{2}}\otimes\dgal[0]{1}\rightarrow\dgal{1}$ being \ee{almost free}, i.e.\@ 
		\begin{eqn}\dgal[*]{1}\cong\lp\dgal[*]{2}\underset{\dgal[0]{2}}\otimes\dgal[0]{1}\rp
			\underset{\dgal[0]{1}}\otimes\dga[*],\end{eqn}
		where $\dga[*]$ is a free $\nopos$-graded $\dgal[0]{1}$-algebra generated by a sequence of finitely generated projective $\dgal[0]{1}$-modules.\end{enumerate}\end{definition}

\begin{remark} It is obvious that $\dgs\in\Dgs[\F]$ is in $\Dman[\F]$, if and only if $\cs$ is quasi-projective over $\F$ and $\dgs\rightarrow\spec{\F}$ is a fibration.
	
It is also clear that $\spec{\dgal{1}}\rightarrow\spec{\dgal{2}}$ in $\Dmana[\F]$ is a fibration, if and only if $\dgal[0]{2}\rightarrow\dgal[0]{1}$ is smooth and there are finitely generated projective $\dgal[0]{1}$-, $\dgal[0]{2}$-modules $\lf\genbu[k]{1}\rf_{k<0}$, $\lf\genbu[k]{2}\rf_{k<0}$ that freely generate $\dgal[*]{1}$, $\dgal[*]{2}$ over $\dgal[0]{1}$ and $\dgal[0]{2}$ respectively, and s.t.\@ $\dgal[*]{2}\underset{\dgal[0]{2}}\otimes\dgal[0]{1}\rightarrow\dgal[*]{1}$ is given by a sequence of inclusions $\lf\genbu[k]{2}\underset{\dgal[0]{2}}\otimes\dgal[0]{1}\hooklongrightarrow\genbu[k]{1}\rf_{k<0}$ with the quotients being projective $\dgal[0]{1}$-modules. In particular $\dgal{2}\underset{\dgal[0]{2}}\otimes\dgal[0]{1}\rightarrow\dgal{1}$ is almost free.
	
\hide{Indeed, since $\dgal[*]{1}$ is freely generated over $\F$ by a sequence $\lf\w{\genbu[k]{1}}\rf_{k<0}$ of finitely generated projective $\dgal[0]{1}$-modules, we have $\dgal[<0]{1}\cong\underset{k<0}\bigoplus\lp\mathfrak I^k\oplus\w{\genbu[k]{1}}\rp$, where $\mathfrak I^k$ is the degree $k$ component of $\lp\dgal[<0]{1}\rp^2$. Let $\lf\genbu[k]{2}\rf_{k<0}$ be a sequence of finitely generated projective $\dgal[0]{2}$-modules that freely generate $\dgal[*]{2}$ over $\F$. Composing $\dgal{2}\underset{\dgal[0]{2}}\otimes\dgal[0]{1}\rightarrow\dgal{1}$ with the projections onto $\lf\w{\genbu[k]{1}}\rf_{k<0}$ we obtain a sequence of bundle maps $\lf\genbu[k]{2}\underset{\dgal[0]{2}}\otimes\dgal[0]{1}\rightarrow\w{\genbu[k]{1}}\rf_{k<0}$. Since $\dgal{2}\rightarrow\dgal{1}$ corresponds to a fibration, locally these bundle maps are inclusions with locally free quotients, therefore also globally these are inclusions of bundles with the quotients being locally free. For each $k<0$ we define $\genbu[k]{1}\subseteq\dgal[k]{1}$ to be the sum of the image of $\genbu[k]{2}\underset{\dgal[0]{2}}\otimes\dgal[0]{1}$ in $\dgal[k]{1}$ and a complement of the image of $\genbu[k]{2}\underset{\dgal[0]{2}}\otimes\dgal[0]{1}$ in $\w{\genbu[k]{1}}$. This sum is clearly direct, i.e.\@ each $\genbu[k]{1}$ is a bundle of the same rank as $\w{\genbu[k]{1}}$, and, working locally, one sees that $\dgal[*]{1}$ is freely generated over $\F$ by $\lf\genbu[k]{1}\rf_{k<0}$.}\end{remark}

\begin{proposition}\label{WrongCatF} Pullbacks of fibrations in $\Dman[\F]$ exist, and are fibrations themselves. If, in addition, the fibrations are quasi-isomorphisms, their pullbacks are quasi-isomorphisms as well.\end{proposition}
\begin{proof} A standard argument shows that $\Dgs[\F]$ has all fiber products (e.g.\@ \cite{Hartshorne} \S II.3), so we only need to show that a pullback of a fibration in $\Dman[\F]$ is a fibration itself and lies in $\Dman[\F]$. To show that it is a fibration it is enough to work in $\Dmana[\F]$, since the property of being a fibration is local on the domain. Let $\dgal{1}\rightarrow\dgal{2}$, $\dgal{1}\rightarrow\dgal{3}$ be morphisms in $\sderca[\F]$, s.t.\@ $\dgal{1}\rightarrow\dgal{2}$ corresponds to a fibration, and let $\dgal{4}:=\dgal{3}\underset{\dgal{1}}\otimes\dgal{2}$. Clearly $\dgal[0]{3}\rightarrow\dgal[0]{4}$ is smooth and
	\begin{eqn}\dgal{4}\cong\lp\dgal{3}\underset{\dgal[0]{3}}\otimes\dgal[0]{4}\rp\underset{\dgal{1}\underset{\dgal[0]{1}}\otimes\dgal[0]{4}}\otimes
		\lp\dgal{2}\underset{\dgal[0]{2}}\otimes\dgal[0]{4}\rp.\end{eqn}
By assumption $\dgal{1}\underset{\dgal[0]{1}}\otimes\dgal[0]{2}\rightarrow\dgal{2}$ is almost free, therefore, since pushouts of almost free morphisms are themselves almost free, we have that $\dgal{3}\underset{\dgal[0]{3}}\otimes\dgal[0]{4}\rightarrow\dgal{4}$ is almost free, i.e.\@ $\dgal{3}\rightarrow\dgal{4}$ corresponds to a fibration. 
	
Since fiber products of quasi-projective schemes are quasi-projective, we conclude that a pullback of a fibration in $\Dman[\F]$ is itself in $\Dman[\F]$.

\smallskip

Now we consider pullbacks of fibrations that are simultaneously quasi-isomorphisms. Let $\dgs_2\rightarrow\dgs_1$ be such a fibration, and $\dgs_3\rightarrow\dgs_1$ be any morphism in $\Dman[\F]$. To show that $\dgs_4:=\dgs_3\underset{\dgs_1}\times\dgs_2\rightarrow\dgs_3$ is a quasi-isomorphism it is enough to prove that $\hofu[0]\lp\dgs_4\rp\rightarrow\hofu[0]\lp\dgs_3\rp$ is an isomorphism of schemes,\footnote{Recall from (\ref{HomologyScheme}) that $\hofu$ is the associated graded scheme functor.} and the morphisms between the stalks of $\streaf[*]{\hofu\lp\dgs_3\rp}$, $\streaf[*]{\hofu\lp\dgs_4\rp}$ are isomorphisms of graded rings. The first statement immediately follows from the fact that $\dgs\mapsto\hofu[0]\lp\dgs\rp$ is a truncation functor from $\Dgs[\F]$ to the category of classical schemes, and has a left adjoint, hence it preserves pullbacks. 

As the second statement is local on the domain we can assume that $\dgs_1$, $\dgs_2$, $\dgs_3$, $\dgs_4\in\Dmana[\F]$. We need to show that $\dgs_4\rightarrow\dgs_3$ corresponds to a quasi-isomorphism of dg algebras. We use the fact that $\dgal{3}\rightarrow\dgal{4}$ is a quasi-isomorphism, if and only if $\ho[0]\lp\dgal{3}\rp\overset\cong\longrightarrow\ho[0]\lp\dgal{4}\rp$ and the relative cotangent complex is acyclic (e.g.\@ \cite{GaiRoz17} Prop.\@ 8.3.2 p.\@ 62). 

Since $\dgal{1}$, $\dgal{2}$, $\dgal{3}$, $\dgal{4}\in\sderca[\F]$ we can compute their cotangent complexes using the usual Kähler differentials (e.g.\@ \cite{BSY} Prop.\@ 1). Since $\dgal{1}\rightarrow\dgal{2}$, $\dgal{3}\rightarrow\dgal{4}$ correspond to fibrations the morphisms 
\begin{eqn}\label{Kaehler}\Omega_{\F}\lp\dgal{2}\rp\underset{\dgal[0]{1}}\otimes\dgal[0]{2}
	\longrightarrow\Omega_{\F}\lp\dgal{2}\rp,\quad
	\Omega_{\F}\lp\dgal{3}\rp\underset{\dgal[0]{3}}\otimes\dgal[0]{4}
	\longrightarrow\Omega_{\F}\lp\dgal{4}\rp\n\end{eqn}
are cofibrations between cofibrant objects in the projective model structures on the categories of complexes of $\dgal[0]{2}$- and $\dgal[0]{4}$-modules. Therefore the standard quotients of (\ref{Kaehler}) are homotopy quotients, i.e.\@ $\Omega_{\dgal{1}}\lp\dgal{2}\rp$ and $\Omega_{\dgal{3}}\lp\dgal{4}\rp$ are the relative cotangent complexes. Therefore, since $\dgal{1}\rightarrow\dgal{2}$ is a quasi-isomorphism, $\Omega_{\dgal{1}}\lp\dgal{2}\rp$  is acyclic. Since pullbacks preserve relative Kähler differentials, also $\Omega_{\dgal{3}}\lp\dgal{4}\rp$ is acyclic. Thus we conclude that the relative cotangent complex of $\dgal{3}\rightarrow\dgal{4}$ is acyclic.\end{proof}

Now we define weak equivalences in $\Dman[\F]$, and we put a stronger condition, than just being a quasi-isomorphism. The issue here is that a dg manifold $\dgs$ can be quasi-isomorphic to an affine dg manifold, yet $\cs$ might not be affine itself. 

\begin{definition}\label{NewDefinition} A dg manifold $\dgs$ is \ee{almost affine}, if there is an affine dg manifold $\dgs'$ and a quasi-isomorphism $\dgs'\rightarrow\dgs$. A morphism $\dgs_1\rightarrow\dgs_2$ in $\Dman[\F]$ is \ee{a weak equivalence}, if it is a quasi-isomorphism and for any affine dg manifold $\dgs$ and any fibration $\dgs\rightarrow\dgs_2$, the pullback $\dgs\underset{\dgs_2}\times\dgs_1$ is almost affine. A morphism in $\Dman[\F]$ that is simultaneously a fibration and a weak equivalence will be called \ee{a trivial fibration}.	\end{definition}

\begin{remark} A quasi-isomorphism $\dgs_1\rightarrow\dgs_2$, s.t.\@ the morphism between the ambient classical schemes $\cs_1\rightarrow\cs_2$ is affine, is clearly a weak equivalence. In particular, any quasi-isomorphism $\dgs_1\rightarrow\dgs_2$, with $\dgs_1$ being affine, is a weak equivalence.\end{remark}

\begin{proposition}\label{DiaFa} For any dg manifold $\dgs$ there is a factorization of the diagonal $\Delta\colon\dgs\rightarrow\dgs\times\dgs$ into a weak equivalence followed by a fibration.\end{proposition}
\begin{proof} We use the standard argument of ``killing cocycles'': let $\dgs_0:=\dgs\times\dgs$ and suppose we have constructed  
		\begin{eqn}\dgs\rightarrow\dgs_k\rightarrow\dgs_{k+1}\rightarrow\ldots\rightarrow\dgs_0=\dgs\times\dgs\end{eqn} 
in $\Dman[\F]$ that factorize $\Delta\colon\dgs\rightarrow\dgs\times\dgs$ and s.t.\@\begin{enumerate}
	\item $\dgs_i\rightarrow\dgs_{i+1}$ is a fibration for each $i\geq k$,
	\item $\hofu\lp\dgs\rp\rightarrow\hofu\lp\dgs_k\rp$ is a closed immersion, i.e.\@ it is defined by a sheaf of dg ideals $\lf\mathfrak I^m\subseteq\streaf[m]{\hofu\lp\dgs_k\rp}\rf_{m\leq 0}$,
	\item $\hofu[>k]\lp\dgs\rp\overset{\cong}\longrightarrow\hofu[>k]\lp\dgs_k\rp$ i.e.\@ $\lf\mathfrak I^m=0\rf_{m>k}$.\end{enumerate} 
Since $\cs_k$ is smooth, separated and of finite type over $\F$, every coherent sheaf on $\cs_k$ is a quotient of a locally free coherent sheaf (\cite{Borelli67} Thm.\@ 3.3). Applying this to 
	\begin{eqn}\w{\mathfrak I}^k:=Z\lp\streaf[k]{\dgs_k}\rp\underset{\Delta_*\lp\streaf[k]{\dgs}\rp}\times
	\Delta_*\lp\streaf[k-1]{\dgs}\rp,\end{eqn}
where $Z$ stands for the sub-sheaf of cocycles, we can add a locally free coherent sheaf of new generators in degree $k-1$ whose differentials are $\w{\mathfrak I}^k$. This produces $\dgs_{k-1}$. Taking the limit\footnote{Since at the $k$-th step we add new generators in degree $-k$, it is clear that this limit exists in $\Dman[\F]$.} we obtain the desired factorization $\dgs\rightarrow\underset{k\leq 0}\lim\;\dgs_k\rightarrow\dgs\times\dgs$. Clearly $\dgs\rightarrow\underset{k\leq 0}\lim\;\dgs_k$ is a quasi-isomorphism, and since the map on the ambient classical schemes is affine, it is a weak equivalence.\end{proof}

\begin{remark}\label{GoodFactorization} Notice that the ambient classical scheme in the path object $\underset{k\leq 0}\lim\,\dgs_k$ constructed above is just $\cs\times\cs$. Having proved that pullbacks in $\Dman[\F]$ preserve fibrations and fibrations that are simultaneously quasi-isomorphisms, we can use the factorization lemma in \cite{Brown73} \S I to conclude that any $\dgs_1\rightarrow\dgs_2$ in $\Dman[\F]$ can be factored into $\dgs_1\rightarrow\dgs'_1\rightarrow\dgs_2$, where $\dgs'_1\rightarrow\dgs_2$ is a fibration and $\dgs_1\rightarrow\dgs'_1$ is a weak equivalence having a left inverse that is fibration. Using the path objects we have constructed above, we see from loc.\@ cit.\@ that the ambient classical scheme of $\dgs'_1$ is just $\cs_1\times\cs_2$.\end{remark}

The following proposition shows that a possible obstruction for a dg manifold to be almost affine lies in the set of of ample line bundles on the ambient classical scheme.

\begin{proposition}\label{AmpleEx} Let $\dgs\rightarrow\spec{\dga}$ be a fibration that is also a quasi-isomorphism. Suppose there is an ample line bundle $\twish[1]{\cs}$ on $\cs$, s.t.\@ $\twish[1]{\cs}|_{\hofu[0]\lp\dgs\rp}$ extends to a line bundle on $\Spec\lp\dga[0]\rp$. Then there is $\ddgb{}$ and a fibration $\spec{\ddgb{}}\rightarrow\dgs$ that is also a quasi-isomorphism.\end{proposition}
\begin{proof} We denote $\dgs_0:=\spec{\dga}$ and $Z:=\hofu[0]\lp\dgs\rp\cong\hofu[0]\lp\dgs_0\rp$. Using $\twish[1]{\cs}$ we choose $n\in\posi$ and an immersion $\cs\times\cs_0\hookrightarrow\mathbb P^n_{\F}\times\cs_0$ over $\cs_0$. Let $\iota\colon Z\hookrightarrow\mathbb P^n_{\F}\times Z$ be given by $Z\hookrightarrow\cs$, $Z\hookrightarrow\cs_0$. I.e.\@ $\iota$ is a section of a relative projective space, and as such it corresponds to a surjective morphism of vector bundles
	\begin{eqn}\label{Emb}\F^{\times^{n+1}}\times Z\longrightarrow\twish[1]{\mathbb P^n_{\F}}|_Z.\n\end{eqn}	
Since $Z$ is affine, we can choose a hyperplane section $\sigma$ of $\mathbb P^n_{\F}\times Z\rightarrow Z$ that does not intersect the image of $\iota$. Such a section is given by choosing a section of  (\ref{Emb}). By assumption $\twish[1]{\mathbb P^n_{\F}}|_Z$ extends to a line bundle $L$ on $\cs_0$. Since $\cs_0$ is affine, we can choose $\w{\sigma}\colon L\rightarrow\F^{\times^{n+1}}\times\cs_0$, that restricts to $\sigma$ on $Z\subseteq\cs_0$. Let $U\subseteq\cs_0$ be the open subscheme where $\w{\sigma}$ is injective, on $U$ $\w{\sigma}$ defines a hyperplane section of $\mathbb P^n_{\F}\times U$. Let $V\subseteq \mathbb P^n_{\F}\times U$ be the complement of this hyperplane. Clearly $Z\subseteq V$, i.e.\@ 
	\begin{eqn}\label{Qinc}\dgs':=\lp V,\streaf{\dgs}\big\vert_{V}\rp\hooklongrightarrow\dgs\n\end{eqn} 
is a quasi-isomorphism, with the morphism between the ambient classical schemes being an open immersion, in particular (\ref{Qinc}) is a fibration. By construction $\cs'$ is quasi-affine over $\cs_0$, and hence it is quasi-affine over $\F$. To go from quasi-affine to affine we will need the following lemma.

\begin{lemma} Given $\dgs'\in\Dgs[\F]$ s.t.\@ $\cs'$ is quasi-affine and $\hofu[0]\lp\dgs'\rp$ is affine over $\F$, there is an open $W\subseteq\cs'$, s.t.\@ $\lp W,\streaf{\dgs'}\big\vert_W\rp\hookrightarrow\dgs'$ is a weak equivalence and $W$ is an affine scheme over $\F$.\end{lemma}
\begin{proof} Consider the open immersion $\cs'\hookrightarrow\spe{\Gamma\lp\cs',\streaf{\cs'}\rp}$ and let $\mathfrak I$ be the ideal of the reduced complement. Let $\lf f_i\rf_{i\in I}$ be some generators of $\mathfrak I$, and let $\lf \overline{f}_i\rf_{i\in I}$ be their images in $\Gamma\lp\hofu[0]\lp\dgs'\rp,\streaf{\hofu[0]\lp\dgs'\rp}\rp$. Since $\hofu[0]\lp\dgs'\rp$ is affine $\lf \overline{f}_i\rf_{i\in I}$ generate the unit ideal in $\Gamma\lp\hofu[0]\lp\dgs'\rp,\streaf{\hofu[0]\lp\dgs'\rp}\rp$, and we can find $\lf g_i\rf_{i\in I}\subset\Gamma\lp\cs',\streaf{\cs'}\rp$, with only finitely many of them $\neq 0$, s.t.\@ $h:=\underset{i\in I}\sum g_if_i$ is invertible around $\hofu[0]\lp\dgs'\rp$. Let $W:=\lf h\neq 0\rf\subseteq\cs'$. It is clearly affine and $\lp W,\streaf{\dgs'}\big\vert_W\rp\hookrightarrow\dgs'$ is a weak equivalence.\end{proof}
Denoting $\spec{\ddgb{}}:=\lp W,\streaf{\dgs'}\big\vert_W\rp$ we have the composite
	\begin{eqn}\spec{\ddgb{}}\hooklongrightarrow\dgs'\hooklongrightarrow\dgs\end{eqn}
which is a fibration and a quasi-isomorphism.\end{proof}

\begin{remark}\label{Sim} It is obvious that given a quasi-isomorphism $\dgs_1\rightarrow\dgs_2$, with $\dgs_1$ being almost affine, also $\dgs_2$ is almost affine. On the other hand, if $\dgs_1\rightarrow\dgs_2$ is a weak equivalence and $\dgs_2$ is affine, $\dgs_1$ has to be almost affine.\end{remark}
Here are some more simple consequences of the definitions.

\begin{proposition}\label{Simple}\begin{enumerate} \item Any quasi-isomorphism $\dgs_1\rightarrow\dgs_2$, where $\dgs_1$ is almost affine, is a weak equivalence.
\item Let $\dgs$ be an almost affine dg manifold. There is an affine dg manifold $\spec{\ddgb{}}$ and a trivial fibration $\spec{\ddgb{}}\rightarrow\dgs$.
\item A morphism $\dgs_1\rightarrow\dgs_2$ is a weak equivalence, if and only if for any almost affine $\dgs$ and a fibration $\dgs\rightarrow\dgs_2$ the pullback $\dgs\underset{\dgs_2}\times\dgs_1$ is almost affine.\end{enumerate}\end{proposition}
\begin{proof}\begin{enumerate} \item By assumption there is a quasi-isomorphism $\spec{\dga}\rightarrow\dgs_1$. Let $\spec{\ddgb{}}\rightarrow\dgs_2$ be any fibration. Then the ambient classical scheme of $\spec{\ddgb{}}\underset{\dgs_2}\times\spec{\dga}$ is affine, and the quasi-isomorphism $\spec{\ddgb{}}\underset{\dgs_2}\times\spec{\dga}\rightarrow\spec{\ddgb{}}\underset{\dgs_2}\times\dgs_1$ shows that $\spec{\ddgb{}}\underset{\dgs_2}\times\dgs_1$ is almost affine.
		
\item Since $\dgs$ is almost affine, we can find $\dga$ and a quasi-isomorphism $\spec{\dga}\rightarrow\dgs$. Consider the factorization 
	\begin{eqn}\spec{\dga}\longrightarrow\dgs_1\longrightarrow\dgs\end{eqn}
into a quasi-isomorphism followed by a fibration as in Remark \ref{GoodFactorization}, where $\cs_1=\Spec\lp\dga[0]\rp\times\cs$. We also have the fibration $\dgs_1\rightarrow\spec{\dga}$, that is a left inverse to $\spec{\dga}\rightarrow\dgs_1$.	This means that for any line bundle $L$ on $\cs_1$ the restriction $L|_{\hogy[0]\lp\dgs_1\rp}$ extends to $\spec{\dga}$. Applying Proposition \ref{AmpleEx} we obtain a fibration $\spec{\ddgb{}}\rightarrow\dgs_1$ that is also a quasi-isomorphism.

\item Suppose $\dgs_1\rightarrow\dgs_2$ is a weak equivalence, and let $\dgs\rightarrow\dgs_2$ be a fibration, with $\dgs$ being almost affine. Using part (2) we can choose a trivial fibration $\spec{\dga}\rightarrow\dgs$. Then 
	\begin{eqn}\spec{\dga}\underset{\dgs_2}\times\dgs_1\rightarrow\dgs\underset{\dgs_2}\times\dgs_1\end{eqn} 
is a quasi-isomorphism showing that $\dgs\underset{\dgs_2}\times\dgs_1$ is almost affine.\end{enumerate}\end{proof}

\begin{proposition}\label{PullTrivial}\begin{enumerate}\item The class of weak equivalences in $\Dman[\F]$ has the $2$-out-of-$3$ property.
\item Pullbacks of trivial fibrations are trivial fibrations.\end{enumerate}\end{proposition}
\begin{proof}\begin{enumerate}\item Consider a sequence of morphisms
	\begin{eqn}\dgs_1\overset{\alpha}\longrightarrow\dgs_2\overset{\beta}\longrightarrow\dgs_3.\end{eqn}
If both $\alpha$ and $\beta$ are weak equivalences, part (3) of Proposition \ref{Simple} immediately implies that $\beta\circ\alpha$ is also a weak equivalence.

Suppose that $\alpha$ and $\beta\circ\alpha$ are weak equivalences. Let $\dgs\rightarrow\dgs_3$ be a fibration with $\dgs$ being almost affine. By assumption $\dgs\underset{\dgs_3}\times\dgs_1$ is almost affine. The quasi-isomorphism $\dgs\underset{\dgs_3}\times\dgs_1\rightarrow\dgs\underset{\dgs_3}\times\dgs_2$ shows that also $\dgs\underset{\dgs_3}\times\dgs_2$ is almost affine, i.e.\@ $\beta$ is a weak equivalence. 

Suppose that $\beta$ and $\beta\circ\alpha$ are weak equivalences. Let $\dgs\rightarrow\dgs_2$ be a fibration with $\dgs$ being affine. The composite $\dgs\rightarrow\dgs_2\rightarrow\dgs_3$ is not necessarily a fibration, but we can factorize it into a weak equivalence followed by a fibration $\dgs\rightarrow\dgs'\rightarrow\dgs_3$ (Remark \ref{GoodFactorization}). Then we have
	\begin{eqn}\dgs\underset{\dgs_2}\times\dgs_1\longrightarrow\dgs'\underset{\dgs_3}\times\dgs_1
	\longrightarrow\dgs_1,\end{eqn}
with the middle term being almost affine and the left arrow being a quasi-isomorphism. Using part (2) of Proposition \ref{Simple} we choose an affine $\dgs''$ and a trivial fibration $\dgs''\rightarrow\dgs'\underset{\dgs_3}\times\dgs_1$. We denote
	\begin{eqn}\dgs''':=\lp\dgs\underset{\dgs_2}\times\dgs_1\rp\underset{\dgs'\underset{\dgs_3}\times\dgs_1}\times
	\dgs''.\end{eqn}
Since $\dgs$ is affine, the pullback of an ample sheaf from $\cs_1$ is ample on the ambient scheme of $\dgs\underset{\dgs_2}\times\dgs_1$. The same is true for $\cs'''$. This pulling back goes through $\cs''$, i.e.\@ we can apply Proposition \ref{AmpleEx} to $\dgs'''\rightarrow\dgs''$, obtaining that $\dgs'''$ is almost affine, and hence so is $\dgs\underset{\dgs_2}\times\dgs_1$.

\item Let $\dgs_1\rightarrow\dgs_2$ be a trivial fibration, and let $\dgs_3\rightarrow\dgs_2$ be any morphisms. We know already that $\dgs_3\underset{\dgs_2}\times\dgs_1\rightarrow\dgs_3$ is a fibration and a quasi-isomorphism. Let $\dgs\rightarrow\dgs_3$ be a fibration with $\dgs$ being affine. We need to show that $\dgs\underset{\dgs_2}\times\dgs_1$ is almost affine. Consider a factorization $\dgs\rightarrow\dgs'\rightarrow\dgs_2$ into a weak equivalence followed by a fibration. Then we have
	\begin{eqn}\dgs\underset{\dgs_2}\times\dgs_1\longrightarrow\dgs'\underset{\dgs_2}\times\dgs_1
		\longrightarrow\dgs_1,\end{eqn}
with the middle term being almost affine and the left arrow being a quasi-isomorphism. Using part (2) of Proposition \ref{Simple} we choose an affine $\dgs''$ and a trivial fibration $\dgs''\rightarrow\dgs'\underset{\dgs_2}\times\dgs_1$. We denote
	\begin{eqn}\dgs''':=\lp\dgs\underset{\dgs_2}\times\dgs_1\rp\underset{\dgs'\underset{\dgs_2}\times\dgs_1}\times
	\dgs''.\end{eqn}
Since $\dgs$ is affine, the pullback of an ample sheaf from $\cs_1$ is ample on the ambient scheme of $\dgs\underset{\dgs_2}\times\dgs_1$. The same is true for $\cs'''$. This pulling back goes through $\cs''$, i.e.\@ we can apply Proposition \ref{AmpleEx} to $\dgs'''\rightarrow\dgs''$, obtaining that $\dgs'''$ is almost affine, and hence so is $\dgs\underset{\dgs_2}\times\dgs_1$.
\end{enumerate}\end{proof}

Recall (\cite{Brown73}, \S I.1) that \ee{a category of fibrant objects} is a category with distinguished sets of weak equivalences and fibrations s.t.\@: there is a final object; every morphism to the final object is a fibration; the class of weak equivalences contains all isomorphisms and has the 2-out-of-3 property; the class of fibrations contains all isomorphisms and is closed with respect to compositions; pullbacks of (trivial) fibrations exist and are (trivial) fibrations; for any object the diagonal morphism can be factored into a weak equivalence followed by a fibration.

\begin{theorem}\label{FibrantObjects} Together with fibrations and weak equivalences as defined above,  $\Dman[\F]$, $\Dmana[\F]$ are categories of fibrant objects.\end{theorem}
\begin{proof} The only conditions that are not obvious are those concerning pullbacks, factorizations and the $2$-out-of-$3$ property. Propositions \ref{WrongCatF}, \ref{DiaFa}, \ref{PullTrivial} prove all these claims.\end{proof}

\begin{remark} Notice that $\Dman[\F]$ has another structure of a category of fibrant objects: the same fibrations, but quasi-isomorphisms instead of weak equivalences. These two structures coincide when restricted to $\Dmana[\F]$, since every quasi-isomorphism between affine dg manifolds is a weak equivalence.\end{remark}

We have 3 categories with distinguished sets of weak equivalences: $\op{\derca[\F]}$, $\Dmana[\F]\cong\op{\sderca[\F]}$ and $\Dman[\F]$. We would like to compare their simplicial localizations (\cite{DK1} \S4). Recall that simplicial localization of a category is a category enriched in $\sset$ that has the same set of objects and realizes the derived functor of universally inverting the weak equivalences. We will call the simplicial sets of morphisms in $\sset$-categories \ee{mapping spaces}, and denote them by $\mapis\lp-,-\rp$.

\begin{definition} (\cite{HAG1} Def.\@ 2.1.3) A $\sset$-functor $\mathcal F\colon\mathcal C\rightarrow\mathcal D$ between $\sset$-categories is \ee{homotopically full and faithful} if 
	\begin{eqn}\forall C_1,C_2\in\mathcal C,\quad
		\mapis[\mathcal C]\lp C_1,C_2\rp\longrightarrow
		\mapis[\mathcal D]\lp\mathcal F(C_1),\mathcal F(C_2)\rp\end{eqn} 
	is a weak equivalence in $\sset$.\end{definition}

\begin{theorem}\label{ManifoldsInclusion} Let $\loca{\Dmana[\F]}$, $\loca{\Dman[\F]}$ be the simplicial localizations. The $\sset$-functor $\loca{\Dmana[\F]}\rightarrow\loca{\Dman[\F]}$, given by $\Dmana[\F]\hookrightarrow\Dman[\F]$, is homotopically full and faithful.\end{theorem}
\begin{proof} Both $\Dmana[\F]$ and $\Dman[\F]$ are categories of fibrant objects (with the same notions of fibrations and weak equivalences), hence they admit the right homotopy calculus of fractions (\cite{Low} Thm.\@ A.5). Therefore to compute (up to weak equivalences) the mapping spaces in simplicial localizations we can use cocycles: recall (\cite{Low} Def.\@ 3.1) that \ee{a cocycle} from $\dgs_1$ to $\dgs_3$ is a diagram $\dgs_1\leftarrow\dgs_2\rightarrow\dgs_3$ where $\dgs_1\leftarrow\dgs_2$ is a trivial fibration. A \ee{morphism between cocycles} is a commutative diagram
	\begin{eqn}\label{ACocycle}\begin{tikzcd} & \dgs_2\ar[ld]\ar[rd]\ar[dd] & \\
			\dgs_1 & & \dgs_3\\
			& \dgs'_2\ar[lu]\ar[ru] & \end{tikzcd}\n\end{eqn}
where $\dgs_2\rightarrow\dgs'_2$ is a weak equivalence. The mapping space from $\dgs_1$ to $\dgs_3$ in $\loca{\Dman[\F]}$ is weakly equivalent to the nerve of the category of cocycles from $\dgs_1$ to $\dgs_3$ (\cite{Low} Thm.\@ 3.12), similarly for $\loca{\Dmana[\F]}$. Moreover these weak equivalences are canonically given by the inclusion of cocycles into all possible hammocks (\cite{DK2} \S2). We will denote the categories of cocycles by $\cocy[{\Dman[\F]}]\lp\dgs_1,\dgs_3\rp$, $\cocy[{\Dmana[\F]}]\lp\dgs_1,\dgs_3\rp$.
	
	For any $\dgs_1,\dgs_3\in\Dmana[\F]$ there is an obvious inclusion 
	\begin{eqn}\label{AffineCocycles}\nerve{\cocy[{\Dmana[\F]}]\lp\dgs_1,\dgs_3\rp}\subset
		\nerve{\cocy[{\Dman[\F]}]\lp\dgs_1,\dgs_3\rp}\n\end{eqn}
	where $\nerve{-}$ stands for the nerve of a category. We claim that (\ref{AffineCocycles}) is a weak equivalence of simplicial sets. According to \cite{DK3} \S7.2, \S8.2 these simplicial sets are
	\begin{eqn}\hocolim[\dgs_2\in\fibra{\Dmana[\F]}{\dgs_1}]\;\Hom[{\Dmana[\F]}]\lp\dgs_2,\dgs_3\rp,\quad
		\hocolim[\dgs_2\in\fibra{\Dman[\F]}{\dgs_1}]\;\Hom[{\Dman[\F]}]\lp\dgs_2,\dgs_3\rp,\end{eqn}
	where $\fibra{\Dman[\F]}{\dgs_1}$, $\fibra{\Dmana[\F]}{\dgs_1}$ are the categories of trivial fibrations ending in $\dgs_1$, and the homotopy limits are computed in $\sset$. Therefore to prove that (\ref{AffineCocycles}) is a weak equivalence it is enough to show that the inclusion $\fibra{\Dmana[\F]}{\dgs_1}\hookrightarrow\fibra{\Dman[\F]}{\dgs_1}$ is left cofinal (\cite{BKan}, \S XI.9), i.e.\@ $\forall\dgs_2\in\fibra{\Dman[\F]}{\dgs_1}$ the nerve of $\lp\fibra{\Dmana[\F]}{\dgs_1}\rp/\dgs_2$ is contractible.
	
	According to Remark \ref{Sim} and Proposition \ref{Simple} $\forall\dgs_2\in\fibra{\Dman[\F]}{\dgs_1}$  $\exists\dgs\in\Dmana[\F]$ and a trivial fibration $\dgs\rightarrow\dgs_2$. Pulling back over $\dgs\rightarrow\dgs_2$ defines
	\begin{eqn}\lp\fibra{\Dmana[\F]}{\dgs_1}\rp/\dgs_2\longrightarrow
		\lp\fibra{\Dmana[\F]}{\dgs_1}\rp/\dgs,\end{eqn}
	which has a left adjoint (composition with $\dgs\rightarrow\dgs_2$). Thus the nerves of $\lp\fibra{\Dmana[\F]}{\dgs_1}\rp/\dgs_2$ and $\lp\fibra{\Dmana[\F]}{\dgs_1}\rp/\dgs$ are weakly equivalent, but the latter category has a final object -- $\dgs$ itself. Hence the nerve of $\lp\fibra{\Dmana[\F]}{\dgs_1}\rp/\dgs_2$ is contractible.\end{proof}

\section{Comparison with affine derived schemes}
Let $\simca[\F]$ be the category of simplicial commutative unital $\F$-algebras. In other words \@ $\simca[\F]$ is the category of functors from $\Delta^{op}$ to the category of associative commutative unital $\F$-algebras.
The opposite category will be called the category of \ee{affine derived schemes} over $\F$. \hide{For an object $\sima{C}\in\simca{\F}$
	\begin{eqn}\sima{C}=\{C_n,\,\{\partial_i:C_{n+1}\rightarrow C_n\}_{i=0}^{n+1},\,\{\sigma_j:C_n\rightarrow C_{n+1}\}_{j=0}^n\}_{n\geq 0}\end{eqn}
	the corresponding object in $\saff{\F}$ will be denoted by $\spec{\sima{C}}$.}

We consider $\simca[\F]$ equipped with the standard simplicial model structure from \cite{Qu67}, and $\saff[\F]$ with the corresponding opposite model structure. Dold--Kan correspondence  induces a Quillen equivalence $\simca[\F]\simeq\derca[\F]$ (\cite{DoldKan} page 289).\footnote{Recall that $\derca[\F]$ denotes the category of differential non-positively graded algebras.} Correspondingly we have a Quillen equivalence $\saff[\F]\simeq\op{\derca[\F]}$, and hence we will call objects of $\op{\derca[\F]}$ affine derived schemes as well.

\smallskip

We would like to compare simplicial localizations of $\op{\sderca[\F]}\cong{\Dmana[\F]}$ and $\op{\derca[\F]}$, where, as before, $\sderca[\F]$ denotes the full subcategory of $\derca[\F]$ consisting of $\dga$ with $\dga[0]$ being smooth over $\F$ and $\dga$ being almost free over $\dga[0]$. 

Since ${\derca[\F]}$ is a model category, with weak equivalences \hide{$\we[{\derca[\F]}]$}being quasi-isomorphisms and fibrations \hide{$\fib[{\derca[\F]}]$}being maps that are surjective in negative degrees, we can use simplicial and cosimplicial resolutions to compute mapping spaces (\cite{DK3} \S5, \cite{Hirsh} \S17). Given $\dgs,\dgs'$ in $\op{\derca[\F]}$, one constructs a cosimplicial object $\lf\dgs[n]\rf_{n\geq 0}\in\lp\op{\derca[\F]}\rp^\Delta$ with special properties (recalled below) and obtains 
\begin{eqn}	\mapis[\loca{\op{\derca[\F]}}]\lp\dgs,\dgs'\rp\simeq
	\lf\Hom[{\op{\derca[\F]}}]\lp\dgs[n],\dgs'\rp\rf_{n\geq 0},\end{eqn}
assuming that $\dgs'$ is fibrant in $\op{\derca[\F]}$. Moreover, these weak equivalences are canonically chosen, if one uses hammock localization for $\loca{\op{\derca[\F]}}$ (\cite{DK3} \S7). This theory does not give us a similar result within $\Dmana[\F]$ because the notions of fibrations in $\Dmana[\F]$ and $\op{\derca[\F]}$ are different: they are just smooth maps in the former case, while they need to be retracts of vector bundles in the latter case.

However, one can produce cosimplicial resolutions with respect to the model structure on $\op{\derca[\F]}$ completely within $\Dmana[\F]$, moreover, these resolutions are functorial. This statement is the contents of the following theorem, and later it will be essential for proving that the mapping spaces in $\loca{\op{\derca[\F]}}$ and in $\loca{\Dmana[\F]}$ are naturally weakly equivalent.	

\begin{theorem}\label{FunctorialResolutions} There is a functor 
	\begin{eqn}\reso\colon\Dmana[\F]\longrightarrow\Dmana[\F]^{\Delta}\end{eqn} 
	that, composed with $\Dmana[\F]^{\Delta}\hookrightarrow\lp\op{\derca[\F]}\rp^{\Delta}$, produces special cosimplicial resolutions (\cite{DK3} \S6.8) of objects of $\Dmana[\F]$ with respect to the model structure on $\op{\derca[\F]}$.\end{theorem}
\begin{proof} To avoid unnecessary dualization we work in $\derca[\F]$ instead of $\op{\derca[\F]}$, and construct special functorial \e{simplicial} resolutions of dg algebras. Given $\dga\in\sderca[\F]\cong\op{\Dmana[\F]}$ we produce a simplicial object\footnote{Here $\sigma_{n+1,i}\colon\oset{n}\hookrightarrow\oset{n+1}$ is the inclusion that misses $i$, with $\sigma^*_{n+1,i}$ being the corresponding \e{face map} $\ddga{n+1}\rightarrow\ddga{n}$, and $\tau_{n,j}\colon\oset{n+1}\rightarrow\oset{n}$ is the weakly order preserving surjection mapping $j,j+1\mapsto j$, with $\tau^*_{n,j}$ being the corresponding \e{degeneracy map} $\ddga{n}\rightarrow\ddga{n+1}$. We will sometimes suppress the indices in the notation of face and degeneracy maps.}
	\begin{eqn}\lf\ddga{n},\lf\sigma^*_{n+1,i}\rf_{0\leq i\leq n+1},\lf\tau^*_{n,j}\rf_{0\leq j\leq n}\rf_{n\geq 0}\end{eqn}
in $\sderca[\F]$ that is functorial in $\dga$ and has the properties of a special simplicial resolution of $\dga$ with respect to the model structure on $\derca[\F]$. 
	
The conditions that are required for a simplicial object to be a special simplicial resolution involve \ee{the matching objects} $\lf\maspa{n}\rf_{n\geq 0}$ and the \ee{latching objects} $\lf\laspa{n}\rf_{n\geq 0}$ (\cite{DK3} \S4.3, \S6.7, \cite{Hovey} Def.\@ 5.2.2). Recall that $\forall n\geq 1$ $\maspa{n}$ of $\lf\ddga{m}\rf_{m\geq 0}$ is the limit (computed in $\derca[\F]$) of the diagram of $n+1$ copies of $\ddga{n-1}$ (indexed by the $n-1$-faces of the $n$-simplex $\Delta_n$) and $\frac{n(n+1)}{2}$ copies of $\ddga{n-2}$ (indexed by the $n-2$-faces of $\Delta_n$), with the morphisms being the corresponding face maps $\ddga{n-1}\rightarrow\ddga{n-2}$. When $n=1$ we set $\ddga{-1}:=0$ -- the final object in $\derca[\F]$.
	
For each $n\geq 1$ the latching object $\laspa{n}$ of $\lf\ddga{m}\rf_{m\geq 0}$ is the colimit of $n$-copies of $\ddga{n-1}$ (indexed by $\lf\tau^*_{n-1,j}\rf_{0\leq j\leq n-1}$) and $\frac{n(n-1)}{2}$ copies of $\ddga{n-2}$ (indexed by the weakly order preserving surjections $\oset{n}\rightarrow\oset{n-2}$), with the morphisms being the corresponding degeneracy maps $\ddga{n-2}\rightarrow\ddga{n-1}$.
	
There is a natural morphism $\laspa{n}\rightarrow\maspa{n}$, and \ee{the conditions on the special simplicial resolutions} require that $\forall n\geq 0$ $\ddga{n}$ provides a factorization $\laspa{n}\longrightarrow\ddga{n}\longrightarrow\maspa{n}$ into a trivial cofibration followed by a fibration, where we set $\laspa{0}:=\dga$ and $\maspa{0}:=0$ (\cite{DK3} \S6.8). 
	
\smallskip
	
We construct $\reso[\dga]$ by induction on the simplicial dimension (the subscript in our notation). In simplicial dimension $0$ there are two conditions: we need to have a trivial cofibration $\dga\overset{\simeq}\longrightarrow\ddga{0}$ that is functorial in $\dga$, and $\ddga{0}$ has to be fibrant with respect to the model structure on $\derca[\F]$. Since every object in $\derca[\F]$ is fibrant we define
	\begin{eqn}\ddga{0}:=\dga,\quad\dga\overset{=}\longrightarrow\ddga{0}.\end{eqn}
Given $n\geq 0$ suppose we have constructed 
	\begin{eqn}\lf\ddga{m},\lf\sigma^*_{m+1,i}\rf_{0\leq i\leq m+1},\lf\tau^*_{m,j}\rf_{0\leq j\leq m}\rf_{0\leq m\leq n},\end{eqn} 
s.t.\@ $\forall m\leq n$ the unique degeneracy map $\ddga{0}\rightarrow\ddga{m}$ is isomorphic to
	\begin{eqn}\label{UniqueDeg}\ddga{0}\hooklongrightarrow\dga\underset{\dga[0]}\otimes
		\fral{\dga[0]}{\underset{k<0}\bigoplus\lp\denbu[k]{m}\overset{\delta}\rightarrow\genbu[k+1]{m}\rp},\n
	\end{eqn}
where $\denbu[k]{m}\overset{\delta}\rightarrow\genbu[k+1]{m}$ is an acyclic complex of finitely generated projective $\dga[0]$-modules with $\denbu[k]{m}$ sitting in degree $k$ and $\genbu[k+1]{m}$ in degree $k+1$, and $\fral{\dga[0]}{-}$ stands for the functor of the free dg commutative $\dga[0]$-algebra generated by a given complex. As part of our inductive assumption we require that (\ref{UniqueDeg}) is functorial in $\dga$.  It is clear that $\ddga{0}$ satisfies this requirement (in the trivial way). 
	
To perform the induction step we need to have a good way to compute matching objects. For each $0\leq j\leq n$ we write $\ske{n+1}{j}$ to mean \ee{the $j$-skeleton} of $\lf\ddga{m}\rf_{0\leq m\leq n}$, i.e.\@ the limit (computed in $\derca[\F]$) of the diagram of copies of $\ddga{j}$, $\ddga{j-1}$ one for each $j$-face, and $j-1$-face of $\Delta_{n+1}$ respectively, with the morphisms being all possible face maps. Clearly $\maspa{n+1}$ is just $\ske{n+1}{n}$, and we have a recursive relation among the skeletons: $\ske{n+1}{j}$ is the pullback (computed in $\derca[\F]$)
	\begin{eqn}\begin{tikzcd} \ske{n+1}{j} \ar[r]\ar[d] & \prod\ddga{j}\ar[d] \\
			\ske{n+1}{j-1}\ar[r] &  \prod\maspa{j},\end{tikzcd}\n\end{eqn}
where the direct products are indexed by the $j$-faces of $\Delta_{n+1}$.
	
For each $0\leq m\leq n$ we denote by 
	\begin{eqn}\underset{k<0}\bigoplus\;\henbu[k]{m}\subseteq\underset{k<0}\bigoplus\;\ddga[k]{m}\end{eqn}
the $\mathbb Z_{<0}$-graded $\dga[0]$-subalgebra generated by $\lf\ddga[k]{0}\oplus\denbu[k]{m}\oplus\genbu[k]{m}\rf_{k<0}$. It is clear that $\forall k<0$ the $\dga[0]$-module $\henbu[k]{m}$ is projective and finitely generated. Then $\forall k<0$ we define $\denbub[k]{n+1}$ to be the pullback of $\dga[0]$-modules
	\begin{eqn}\label{NewGenerators}\begin{tikzcd} \denbub[k]{n+1} \ar[r]\ar[d] & 
			\underset{0\leq i\leq n+1}\prod\henbu[k]{n}\ar[d] \\
			\ske{n+1}{n-1}\ar[r] &  \underset{0\leq i\leq n+1}\prod\maspa{n},
		\end{tikzcd}\n\end{eqn}
where the direct products are indexed by the $n$-faces of $\Delta_{n+1}$. Using the unique degeneracy map $\ddga{0}\rightarrow\laspa{n+1}$ we define
	\begin{eqn}\ddga{n+1}:=\laspa{n+1}\underset{\dga[0]}\otimes\fral{\dga[0]}{\underset{k<0}\bigoplus\lp\denbub[k]{n+1}\overset{\delta}\rightarrow\genbub[k+1]{n+1}\rp},\end{eqn}
where $\forall k<0$ $\denbub[k]{n+1}\overset{\delta}\rightarrow\genbub[k+1]{n+1}$ is an acyclic complex of $\dga[0]$-modules concentrated in degrees $k$ and $k+1$. It is obvious that $\ddga{n+1}$ is functorial in $\dga$. We would like to show that $\ddga{n+1}$ has the form (\ref{UniqueDeg}).\footnote{Notice that tensoring with $\laspa{n+1}$ introduces additional projective $\dga[0]$-modules as generators, hence the difference in notation between $\denbu[k]{n+1}$ and $\denbub[k]{n+1}$.} Thus we need to prove that each $\denbub[k]{n+1}$ is a finitely generated projective $\dga[0]$-module, and that $\laspa{n+1}\longrightarrow\ddga{n+1}\longrightarrow\maspa{n+1}$
is a trivial cofibration followed by a fibration (with respect to the model structure on $\derca[\F]$).
	
To prove the first claim we analyze the properties of $\lf\denbu[k]{m}\rf_{k<0}$ for $m\leq n$. When $m=1$ the pullback diagram (\ref{NewGenerators}) becomes $\forall k<0$
	\begin{eqn}\begin{tikzcd} \denbub[k]{1} \ar[r,"\cong"]\ar[d] & 
			\underset{0\leq i\leq 1}\prod\henbu[k]{0}=\underset{0\leq i\leq 1}\prod\ddga[k]{0}\ar[d] \\
			0\ar[r] &  \underset{0\leq i\leq 1}\prod 0.
	\end{tikzcd}\end{eqn}
This allows us to have decompositions
	\begin{eqn}\forall k<0\quad\denbub[k]{1}=\denbu[k]{1}=\underset{s\in\pset{0,1}}\bigoplus\denbu[k]{1,s},\quad
		\genbu[k]{1}=\underset{s\in\pset{0,1}}\bigoplus\genbu[k]{1,s}\end{eqn}
where $\pset{0,1}$ is the set of non-empty subsets of $\lf 0,1\rf$, $\denbu[k]{1,\lf 0\rf}$, $\denbu[k]{1,\lf 1\rf}$ are mapped isomorphically onto the copies of $\ddga[k]{0}$ indexed by the corresponding $0$-simplices in $\Delta_1$, and $\denbu[k]{1,\lf 0,1\rf}$ is the (trivial) kernel of $\denbu[k]{1}\rightarrow\ddga[k]{0}\oplus\ddga[k]{0}$. Then clearly we can have
	\begin{eqn}\forall k<0\quad\henbu[k]{1}=\underset{s\in\pset{0,1}}\bigoplus\henbu[k]{1,s}\end{eqn}
with $\henbu[k]{1,\lf 0\rf}$, $\henbu[k]{1,\lf 1\rf}$ projecting isomorphically onto the copies of $\henbu[k]{0}$, and $\henbu[k]{1,\lf 0,1\rf}$ being the kernel of $\henbu[k]{1}\rightarrow\henbu[k]{0}\,\oplus\,\henbu[k]{0}$. Now we make an addition to our inductive assumption and suppose that for each $0\leq m\leq n$ we have decompositions 
	\begin{eqn}\label{Decompositions}\forall k<0\quad\henbu[k]{m}=
		\underset{s\in\pset{0,1,\ldots,m}}\bigoplus\henbu[k]{m,s},\n\end{eqn}
s.t.\@ $\henbu[k]{m,\lf 0,1,\ldots,m\rf}$ is the kernel of $\henbu[k]{m}\rightarrow\maspa{m}$, and $\forall j<m$ and each order preserving inclusion $\sigma\colon\oset{j}\hookrightarrow\oset{m}$ we have isomorphisms
	\begin{eqn}\label{Isos}\forall k<0\quad 
		\sigma^*\colon\henbu[k]{m,\sigma\lp 0,\ldots,j\rp}\overset{\cong}\longrightarrow
		\henbu[k]{j,\oset{j}}.\n\end{eqn}
Then we obtain an isomorphism of $\dga[0]$-modules
	\begin{eqn}\label{InductiveDefinition}\forall k<0\quad
		\denbub[k]{n+1}\cong\underset{\underset{|s|<n+1}{s\in\pset{0,\ldots,n+1}}}\bigoplus
		\henbu[k]{|s|-1,\oset{|s|-1}},\n\end{eqn}
where $|s|$ denotes the number of elements in $s$. Using induction we obtain then that $\forall k<0$ the $\dga[0]$-module $\denbu[k]{n+1}$ is finitely generated and projective, and moreover $\henbu[k]{n+1}$ satisfies the additional assumption (\ref{Decompositions}). It follows immediately that $\laspa{n+1}\rightarrow\ddga{n+1}$ is a trivial cofibration in $\derca[\F]$.
	
	\smallskip
	
It remains to show that $\ddga{n+1}\rightarrow\maspa{n+1}$ is a fibration in $\derca[\F]$ i.e.\@ it is surjective in negative degrees. We notice that $\forall k<0$ the $\dga[0]$-module $\ddga[k]{n+1}$ is graded by the powers of $\genbu[0]{n+1}$ (with the degree $0$ component being $\henbu[k]{n+1}$), similarly for $\ddga[k]{n}$. Moreover $\ddga[k]{n+1}\rightarrow\maspa[k]{n+1}$ is filtered with respect to these gradings.\footnote{This morphism is only filtered, not necessarily graded, since an element of $\genbu[0]{n+1}$ can be mapped to $\dga[0]\subseteq\ddga[0]{n}$.} So it is enough to check surjectivity on the fibers at the closed points of $\spe{\dga[0]}$, i.e.\@ we can assume that $\dga[0]=\F$.
	
Recall that $\maspa{n+1}$ can be seen as a pullback of dg algebras
	\begin{eqn}\begin{tikzcd} \maspa{n+1} \ar[r]\ar[d] & 
			\underset{0\leq i\leq n+1}\prod\ddga{n}\ar[d] \\
			\ske{n+1}{n-1}\ar[r] &  \underset{0\leq i\leq n+1}\prod\maspa{n},
	\end{tikzcd}\end{eqn}
where the direct products are indexed by the $n$-simplices in $\Delta_{n+1}$. The kernel of the right vertical arrow is the product of $n+2$ copies of the kernel $\kerma[n]$ of $\ddga{n}\rightarrow\maspa{n}$. We claim that $\forall n\geq 1$ and for any $n$-face of $\Delta_{n+1}$ the corresponding summand
	\begin{eqn}\label{KerSummand}0\oplus\ldots\oplus\kerma[n]\oplus\ldots\oplus 0\subseteq
		\underset{0\leq i\leq n+1}\prod\ddga{n}\n\end{eqn}
is in the image of $\ddga{n+1}\rightarrow\maspa{n+1}$. The assumption $\dga[0]=\F$ implies that $\forall m\geq 1$ we can introduce $\oh{\genbu[0]{m}}=\underset{s\in\oset{m}}\bigoplus\oh{\genbu[0]{m,s}}$ having the same properties as in (\ref{Isos}), and s.t.\@ $\ddga[0]{m}\cong\F\oplus\oh{\genbu[0]{m}}$.\footnote{If $1\in\F$ is not a coboundary in the fiber of $\dga$ over the closed point in $\spe{\dga[0]}$, this is true also for $m=0$.}  Hence to prove the claim it is enough to show that $\forall\alpha\in\oh{\genbu[0]{n}}$, $\forall\beta\in\underset{k<0}\bigoplus\;\ddga[k]{n}$, if $\alpha\cdot\beta\in\kerma[n]$, then $0+\ldots+\alpha\cdot\beta+\cdots+0$ is in the image of $\ddga{n+1}\rightarrow\maspa{n+1}$. We have
	\begin{eqn}\alpha=\underset{s\in\pset{\oset{n}}}\sum\alpha_s,\quad
		\beta=\underset{s'\in\pset{\oset{n}}}\sum\beta_{s'},\end{eqn}
where $\alpha_s\in\oh{\genbu[0]{n,s}}$, $\beta_{s'}\in\ddga[k]{n,s'}$, and $\ddga[k]{n}=\underset{s'\in\pset{\oset{n}}}\bigoplus\ddga[k]{n,s'}$ is a decomposition as in (\ref{Decompositions}). The assumption that $\alpha\cdot\beta\in\kerma[n]$ implies that, if $s\cup s'\neq\lf 0,\ldots,n+1\rf$, then at least one of $\alpha_s$, $\beta_{s'}$ is $0$. Now we define
	\begin{eqn}\w{\alpha}:=\underset{s\in\pset{0,\ldots,n+1}}\sum\;\underset{1\leq j\leq n+1-|s|}\sum\alpha_s,\quad
		\w{\beta}:=\underset{s'\in\pset{0,\ldots,n+1}}\sum\;\underset{1\leq j\leq n+1-|s'|}\sum\beta_{s'},\end{eqn}
where the internal sums are over all the $n$-faces of $\Delta_{n+1}$ that contain $s$ (or $s'$ respectively). By construction $\w{\alpha}$, $\w{\beta}$ are in the image of $\ddga{n+1}\rightarrow\maspa{n+1}$, and one immediately obtains that $\w{\alpha}\cdot\w{\beta}=\alpha\cdot\beta$. 
	
Having established that (\ref{KerSummand}) is in the image of $\ddga{n+1}\rightarrow\maspa{n+1}$, we divide by the sum of all copies of $\kerma[n]$, obtaining $\underset{0\leq i\leq n+1}\prod\maspa{n}$. Now we need to prove that in negative degrees $\ddga{n+1}$ maps surjectively onto the image of 
	\begin{eqn}\label{OneDown}\ske{n+1}{n-1}\longrightarrow\underset{0\leq i\leq n+1}\prod\maspa{n}\n\end{eqn}
Again we use induction on $n$. Presenting each copy of $\maspa{n}$ as a pullback
	\begin{eqn}\begin{tikzcd} \maspa{n} \ar[r]\ar[d] & 
			\underset{0\leq i\leq n}\prod\ddga{n-1}\ar[d] \\
			\ske{n}{n-2}\ar[r] &  \underset{0\leq i\leq n}\prod\maspa{n-1}
	\end{tikzcd}\end{eqn}
we apply the same technique of looking at the kernels of the fibration $\ddga{n-1}\rightarrow\maspa{n-1}$ and proving that each copy is in the image of (\ref{OneDown}). Continuing like this we arrive at the image of 
	\begin{eqn}\label{FinalStep}\ske{n+1}{0}=\underset{0\leq i\leq n+1}\prod\ddga{0}\longrightarrow
		\underset{0\leq i\leq n+1}\prod\maspa{n}.\n\end{eqn} 
Since $\genbu[0]{0}=0$, (\ref{InductiveDefinition}) immediately implies that in negative degrees $\ddga{n+1}$ maps surjectively onto the image of (\ref{FinalStep}). \end{proof}

\begin{remark}\label{Retract} Notice that $\forall\dgs\in\Dmana[\F]$ the component $\dgs[0]$ of $\reso[\dgs]=\lf\dgs[n]\rf_{n\geq 0}$ is $\dgs$ itself. Let $\catpo$ be the category with only one object and only the identity endomorphism. Let $\iota\colon\catpo\hookrightarrow\Delta$ be the inclusion as the dimension $0$ object. Then the composite functor
	\begin{eqn}\Dmana[\F]\overset{\reso}\longrightarrow\lp\Dmana[\F]\rp^{\Delta}
		\overset{\iota^*}\longrightarrow\Dmana[\F]\end{eqn}
is the identity.\end{remark}

Now we would like to use the functorial cosimplicial resolution $\reso$ constructed in Theorem \ref{FunctorialResolutions} to show that the mapping spaces in the simplicial localizations of $\Dmana[\F]$ and $\op{\derca[\F]}$ are naturally weakly equivalent. To do this we would like to realize $\reso$ as an endofunctor on $\Dmana[\F]$. Obviously this cannot be done on the nose, so we need to produce a category equivalent to $\Dmana[\F]$ that allows for more space.

By adjunction $\reso$ corresponds to $\Dmana[\F]\times\Delta\rightarrow\Dmana[\F]$. We would like to substitute the codomain of this functor with an equivalent category. Let $\I$ be the category having $\noneg$ as objects, and $\forall m,n\in\noneg$ having exactly one morphism ${m}\rightarrow{n}$. It is clear that the unique $\I\rightarrow\catpo$ together with the inclusion $\catpo\hookrightarrow\I$ as ${0}$ constitute an equivalence of categories. We denote $\dmana[\F]:=\Dmana[\F]\times\I$ and then define $\sreso\colon\Dmana[\F]\times\Delta\longrightarrow\dmana[\F]$
	\begin{eqn}\forall n\geq 0\quad\sreso[\dgs,\lf 0,\ldots,n\rf]:=\lp\dgs[n],n\rp,\end{eqn}
where $\lf\dgs[n]\rf_{n\geq 0}=\reso[\dgs]$. Clearly $\sreso$ is injective on objects. In general $\sreso$ is not faithful, but, as the following Lemma shows, when it matters to us $\sreso$ is faithful. Recall that given $\dgs\in\Dmana[\F]$ the category $\fibra{\Dmana[\F]}{\dgs}$ consists of trivial fibrations $\dgs'\rightarrow\dgs$ (as defined within $\Dmana[\F]$) and weak equivalences between them.

\begin{lemma} Let $\dgs\in\Dmana[\F]$ be s.t.\@ the corresponding $\dga\in\sderca[\F]$ has at least one non-zero element of negative degree. Then
	\begin{eqn}\reso\colon\lp\fibra{\Dmana[\F]}{\dgs}\rp\times\Delta\longrightarrow\fibra{\Dmana[\F]}{\dgs}
	\end{eqn}
	is a faithful functor. Consequently $\sreso\colon\fibra{\Dmana[\F]}{\dgs}\times\Delta\longrightarrow\fibra{\dmana[\F]}{\dgs}$ is an inclusion of a subcategory.\end{lemma}
\begin{proof} Let $\dgs',\dgs''\in\fibra{\Dmana[\F]}{\dgs}$ and let $\ddgb{},\ddgc{}$ be the corresponding dg algebras of functions. Since $\dgs'\rightarrow\dgs$ is a fibration in $\Dmana[\F]$, there are $k<0$ and $0\neq\alpha\in\ddgb[k]{}$, s.t.\@ $\alpha$ is in the image of $\dga[k]\rightarrow\ddgb[k]{}$. Let $\ddgb{\bullet}:=\reso[\ddgb{}]$, $\ddgc{\bullet}:=\reso[\ddgc{}]$. Given $m,n\in\noneg$, two parallel morphisms in $\lp\fibra{\Dmana[\F]}{\dgs}\rp\times\Delta$ are pairs
	\begin{eqn}\lp\phi,\psi\rp,\lp\phi',\psi'\rp\colon\lp\ddgb{},\oset{m}\rp\longrightarrow
		\lp\ddgc{},\oset{n}\rp,\end{eqn}
where $\phi,\phi'\colon\ddgb{}\rightarrow\ddgc{}$ correspond to morphisms $\dgs''\rightarrow\dgs'$ in $\fibra{\Dmana[\F]}{\dgs}$, and $\psi,\psi'\colon\oset{n}\rightarrow\oset{m}$ are weakly order preserving maps.  Suppose $\reso[\phi,\psi]=\reso[\phi',\psi']$. Since there are unique degeneracy maps $\ddgb{0}\rightarrow\ddgb{m}$, $\ddgc{0}\rightarrow\ddgc{n}$ we immediately conclude that $\phi=\phi'$.
	
Suppose there is $i\in\oset{n}$ s.t.\@ $j:=\psi(i)\neq\psi'(i)=:j'$. Using the notation from the proof of Theorem \ref{FunctorialResolutions}, let $\beta\in\henbu[k]{m,\lf j\rf}\subset\ddgb[k]{m}$ be s.t.\@ it is mapped to $\alpha$ by the face map corresponding to $\iota\colon\lf j\rf\hookrightarrow\oset{m}$. By construction the face map corresponding to $\iota'\colon\lf j'\rf\hookrightarrow\oset{m}$ maps $\beta\mapsto 0$. Now let $\w{\psi}$, $\w{\psi'}$ be the compositions of $\psi$ and $\psi'$ with $\lf i\rf\subseteq\oset{n}$. We have
	\begin{eqn}\reso[\phi,\iota^*]=\reso[\phi,\w{\psi}]=\reso[\phi,\w{\psi'}]=\reso[\phi,\iota'^*],\end{eqn}
arriving at a contradiction, since the leftmost term equals $\phi\lp\alpha\rp\in\ddgc[k]{}$, while the rightmost term is $0$.\end{proof}

Having realized $\sreso$ as an inclusion of a subcategory we can compare the mapping spaces in $\loca{\Dmana[\F]}$ and $\loca{\op{\derca[\F]}}$. The key here is functoriality of $\sreso$, that allows us to reduce the computation of mapping spaces in $\loca{\Dmana[\F]}$ to spaces of maps out of cosimplicial resolutions, which are weakly equivalent to mapping spaces computed in $\loca{\op{\derca[\F]}}$.

\begin{theorem}\label{AlgebrasInclusion} Let $\loca{\Dmana[\F]}$, $\loca{\op{\derca[\F]}}$ be the simplicial localizations. The $\sset$-functor $\loca{\Dmana[\F]}\rightarrow\loca{\op{\derca[\F]}}$, given by $\Dmana[\F]\hookrightarrow\op{\derca[\F]}$, is homotopically full and faithful.\end{theorem}
\begin{proof} We need to show that given $\dgs,\dgs'\in\Dmana[\F]$ the map
	\begin{eqn}\mapis[\loca{\Dmana[\F]}]\lp\dgs,\dgs'\rp\longrightarrow
		\mapis[\loca{\op{\derca[\F]}}]\lp\dgs,\dgs'\rp\end{eqn}
is a weak equivalence of simplicial sets. Since in both categories the homotopy types of mapping spaces are invariant with respect to weak equivalences between domains and between codomains, in each case we can choose any representative from a given weak equivalence class of objects. We choose $\dgs$, $\dgs'$ so that their dg algebras of functions $\ddga{}$, $\ddgb{}$ are almost free dg algebras, with $\ddga{}$ having at least one non-zero element of negative degree. 
	
To compute $\mapis[\loca{\Dmana[\F]}]\lp\dgs,\dgs'\rp$ we can use the category of cocycles $\cocy[{\Dmana[\F]}]\lp\dgs,{\dgs'}\rp$, i.e.\@ we can take the nerve of $\lp\fibra{\Dmana[\F]}{\dgs}\rp/\lifto{\dgs'}$, where $\lifto{\dgs'}:=\dgs\times\dgs'$. Equivalently we can use $\lp\fibra{\dmana[\F]}{\dgs}\rp/\lifto{\dgs'}$. Within $\fibra{\dmana[\F]}{\dgs}$ we have the subcategory $\mathfrak D:=\sreso\lp\lp\fibra{\Dmana[\F]}{\dgs}\rp\times\Delta\rp$. Hence we have an inclusion
	\begin{eqn}\mathfrak D/\lifto{\dgs'}\hooklongrightarrow\lp\fibra{\dmana[\F]}{\dgs}\rp/\lifto{\dgs'}.\end{eqn}
To compute $\mapis[\loca{\op{\derca[\F]}}]\lp\dgs,\dgs'\rp$ we can use the category of cocycles $\cocy[\op{\derca[\F]}]\lp\dgs,\dgs'\rp$, i.e.\@ we can take the nerve of $\mathfrak F/\lifto{\dgs'}$, where $\mathfrak F$ is the category of all trivial fibrations onto $\dgs$ in $\op{\derca[\F]}$.\footnote{Here we use the notion of fibrations in $\op{\derca[\F]}$, not in $\Dmana[\F]$.} 
	
Now we compare $\mapis[\loca{\Dmana[\F]}]\lp\dgs,\dgs'\rp$ and $\mapis[\loca{\op{\derca[\F]}}]\lp\dgs,\dgs'\rp$. Introducing additional copies of objects, we can assume that $\mathfrak D\subset\op{\derca[\F]}/\dgs$ and $\mathfrak D\cap\mathfrak F=\emptyset$. We define $\mathfrak D\ltimes\mathfrak F\subset\op{\derca[\F]}/\dgs$ to be the subcategory consisting of $\mathfrak D\cup\mathfrak F$ and all morphisms from objects of $\mathfrak D$ to objects of $\mathfrak F$. We claim that $\mathfrak D\hookrightarrow\mathfrak D\ltimes\mathfrak F$ is a left cofinal functor, i.e.\@ $\forall\dgs''\in\mathfrak F$ the nerve $\nerve{\mathfrak D/\dgs''}$ is contractible. It is here that we use functoriality of $\reso$, observing that 
	\begin{eqn}\mathfrak D=\sreso\lp\lp\fibra{\Dmana[\F]}{\dgs}\rp\times\Delta\rp\cong
		\lp\fibra{\Dmana[\F]}{\dgs}\rp\times\Delta.\end{eqn} 
Therefore we can compute homotopy colimits in two steps: first as a derived functor of $\sset^{\lp\fibra{\Dmana[\F]}{\dgs}\rp\times\Delta}\rightarrow\sset^{\fibra{\Dmana[\F]}{\dgs}}$ (the left adjoint to the pullback over $\lp\fibra{\Dmana[\F]}{\dgs}\rp\times\Delta\rightarrow\fibra{\Dmana[\F]}{\dgs}$) and then as a derived functor of $\sset^{\fibra{\Dmana[\F]}{\dgs}}\rightarrow\sset$. Using these two steps we have
	\begin{eqn}\nerve{\mathfrak D/\dgs''}\simeq\underset{\dgs'''\in\,\fibra{\Dmana[\F]}{\dgs}}\hocolim\;
		\Hom[\op{\derca[\F]}/\dgs]\lp\reso[\dgs'''],\dgs''\rp\simeq\end{eqn}
	\begin{eqn}\simeq\underset{\dgs'''\in\,\fibra{\Dmana[\F]}{\dgs}}\hocolim\;
		\Hom[\op{\derca[\F]}/\dgs]\lp\reso[\dgs'''],\dgs\rp\simeq\nerve{\mathfrak D/\dgs}.\end{eqn}
The first and the last weak equivalences are as in \cite{DK3} \S8.2, while the middle weak equivalence is due to $\dgs''\rightarrow\dgs$ being a trivial fibration in $\op{\derca[\F]}$ (\cite{DK3} \S6.2). Since $\dgs\in\mathfrak D$, it is clear that $\nerve{\mathfrak D/\dgs}$ is contractible.
	
\smallskip
	
Now let $\mathfrak R\colon\mathfrak D\rightarrow\mathfrak F$ be the fibrant replacement functor with respect to the model structure on $\op{\derca[\F]}$. This gives us two functors 
	\begin{eqn}\iota,\mathfrak R\colon\mathfrak D\hooklongrightarrow\mathfrak D\ltimes\mathfrak F\end{eqn} 
and a natural transformation $\iota\rightarrow\mathfrak R$ that consists of weak equivalences. Then the corresponding maps on nerves $\nerve{\mathfrak D}\rightrightarrows\nerve{\mathfrak D\ltimes\mathfrak F}$ are homotopic. Altogether various inclusions of subcategories give us the following diagrams
	\begin{eqn}\Hom[\op{\derca[\F]}]\lp\reso\lp\dgs\rp,\dgs'\rp\rightarrow\nerve{\mathfrak D/\lifto{\dgs'}}\overset{\mathfrak R}\longrightarrow\nerve{\mathfrak F/\lifto{\dgs'}}\rightarrow\nerve{\mathfrak D\ltimes\mathfrak F/\lifto{\dgs'}},\end{eqn}
	\begin{eqn}\Hom[\op{\derca[\F]}]\lp\reso\lp\dgs\rp,\dgs'\rp\rightarrow\nerve{\lp\fibra{\Dmana[\F]}{\dgs}\rp/\lifto{\dgs'}}\rightarrow\nerve{\mathfrak D/\lifto{\dgs'}}\rightarrow\end{eqn}
	\begin{eqn}\rightarrow\nerve{\lp\fibra{\dmana[\F]}{\dgs}\rp/\lifto{\dgs'}}.\end{eqn}
The first two arrows in the first diagram compose to a weak equivalence because $\Hom[\op{\derca[\F]}]\lp\reso\lp\dgs\rp,\dgs'\rp$ computes mapping spaces in $\op{\derca[\F]}$ (\cite{DK3}, \S4.4). The last two arrows in this diagram compose to a weak equivalence because $\mathfrak D\hookrightarrow\mathfrak D\ltimes\mathfrak F$ is a left cofinal functor (\cite{BKan}, \S XI.9). Therefore all maps in this diagram are weak equivalences.\footnote{This is an instance of the 2-out-of-6 property (\cite{DHKS} \S26.2).} Hence the first two arrows in the second diagram compose to a weak equivalence, while the composition of the last two arrows is an isomorphism (Remark \ref{Retract}). Thus we conclude that
	\begin{eqn}\Hom[\op{\derca[\F]}]\lp\reso\lp\dgs\rp,\dgs'\rp\overset{\simeq}\longrightarrow
		\nerve{\lp\fibra{\Dmana[\F]}{\dgs}\rp/\lifto{\dgs'}}.\end{eqn}\end{proof}

\section{Categories of stacks}\label{SectionStacks}
We have considered three categories:\begin{itemize}
	\item $\op{\derca[\F]}$, where $\derca[\F]$ is the category of differential $\nopos$-graded $\F$-algebras,
	\item $\Dmana[\F]\cong\op{\sderca}$ -- the category of affine dg manifolds,
	\item $\Dman[\F]$ -- the category of all dg manifolds. 
\end{itemize}
Theorems \ref{ManifoldsInclusion}, \ref{AlgebrasInclusion} tell us that the inclusions $\Dmana[\F]\hookrightarrow\op{\derca[\F]}$ and $\Dmana[\F]\hookrightarrow\Dman[\F]$ induce homotopically full and faithful functors between the corresponding simplicial localizations. Now we would like to put Grothendieck topologies on these simplicial localizations (\cite{HAG1} \S3.1) and compare the corresponding categories of derived stacks.

\smallskip

On $\loca{\op{\derca[\F]}}$ we use the étale topology (e.g.\@ \cite{ToVa} \S2.3). Recall that $\ddga{}\rightarrow\ddgb{}$ is \ee{an étale covering},\footnote{Instead of specifying étale coverings in the homotopy category $\Ho{\derca[\F]}$ of $\derca[\F]$ we describe their representatives in $\derca[\F]$.} if\begin{enumerate}
	\item $\hogy[0]\lp\ddga{}\rp\rightarrow\hogy[0]\lp\ddgb{}\rp$ is an étale morphism of $\F$-algebras,
	\item $\forall k<0$ $\hogy[k]\lp\ddga{}\rp\underset{\hogy[0]\lp\ddga{}\rp}\otimes\hogy[0]\lp\ddgb{}\rp
	\overset{\cong}\longrightarrow\hogy[k]\lp\ddgb{}\rp$,
	\item $\spe{\hogy[0]\lp\ddgb{}\rp}\longrightarrow\spe{\hogy[0]\lp\ddga{}\rp}$ is a surjective morphism of schemes over $\F$.\end{enumerate}
The first two conditions express the notion of being étale, while the last one specifies what it means for a morphism to be a covering. We define $\ddga{}\rightarrow\ddgb{}$ in $\sderca[\F]$ to be an étale covering, if it is such as a morphism in $\derca[\F]$, i.e.\@ if it satisfies the 3 conditions above. 

\begin{remark}As $\Dmana[\F]\hookrightarrow\op{\derca[\F]}$ induces a homotopically full and faithful $\loca{\Dmana[\F]}\rightarrow\loca{\op{\derca[\F]}}$, when dealing with Grothendieck topologies on $\loca{\Dmana[\F]}$ we can use \ee{the essential image} $\fderca[\F]\subset\derca[\F]$ of $\sderca[\F]$, i.e.\@ the full subcategory consisting of objects that are weakly equivalent to objects in $\sderca[\F]$. One immediately obtains that $\dga\in\fderca[\F]$, if and only if $\hogy\lp\dga\rp$ has finitely many generators in each degree. Then it is clear that, if $\ddga{}\rightarrow\ddgb{}$ is an étale covering in $\derca[\F]$, and one of the objects is in $\sderca$, so is the other.\end{remark}

We use essentially the same notion of covering for $\Dman[\F]$. We say that $\lp\Phi,\smor\rp\colon\dgs\rightarrow\dgs'$ in $\Dman[\F]$ is \ee{an étale covering}, if\begin{enumerate}
	\item $\hofu[0]\lp\dgs\rp\rightarrow\hofu[0]\lp\dgs'\rp$ is an étale morphism of schemes over $\F$,
	\item $\forall k<0$ $\hosh[k]\lp\Phi^{-1}\lp\streaf[\bullet]{\dgs'}\rp\rp
	\underset{\Phi^{-1}\lp\hosh[0]\lp\streaf[\bullet]{\dgs'}\rp\rp}\otimes\hosh[0]\lp\streaf[\bullet]{\dgs}\rp
	\overset{\cong}\longrightarrow\hosh[k]\lp\streaf[\bullet]{\dgs}\rp$,
	\item $\hofu[0]\lp\dgs\rp\longrightarrow\hofu[0]\lp\dgs'\rp$ is a surjective morphism of schemes over $\F$.\end{enumerate}
\begin{remark} Equivalently we can require that for any open affine chart\footnote{For any open $U\subseteq\cs$ we have the dg manifold $\lp U,\streaf[\bullet]{\dgs}|_U\rp\subseteq\dgs$, if it happens to be isomorphic to $\spec{\dga}$, we call it \ee{an open affine chart}.} $\spec{\dga}\subseteq\dgs'$ the morphism $\dgs\underset{\dgs'}\times\spec{\dga}\longrightarrow\spec{\dga}$ belongs to the essential image of $\Dmana[\F]$ in $\Dman[\F]$ and is an étale covering. Clearly the two notions of covering in $\Dmana[\F]$ -- one from $\op{\derca[\F]}$ and one from $\Dman[\F]$ -- coincide.\end{remark}

Now we consider the categories of simplicial prestacks on the $\sset$-categories $\loca{\op{\derca[\F]}}$, $\loca{\Dmana[\F]}$ and $\loca{\Dman[\F]}$ (\cite{HAG1} Def.\@ 2.3.2). We denote them by $\spr{\loca{\op{\derca[\F]}}}$, $\spr{\loca{\Dmana[\F]}}$ and $\spr{\loca{\Dman[\F]}}$ respectively. These are simplicial model categories (\cite{HAG1} \S2.3.1).  In each case we denote the corresponding Yoneda embedding (\cite{HAG1} \S2.4) by putting hats, e.g.\@ $\loca{\Dman[\F]}\ni\dgs\mapsto\yo{\dgs}\in\spr{\loca{\Dman[\F]}}$. The following proposition will be important in proving that open affine charts on dg manifolds constitute a hypercover.

\begin{proposition}\label{TwoIntersections} Let $\dgs,\dgt\in\Dman[\F]$, and let $U\subseteq\cs$ be an open subset. Denote $\dgs':=\lp U,\streaf[\bullet]{\dgs}|_U\rp$ and let $\dgs'\underset{\dgs}\times\dgt$ be the fiber product computed in $\Dman[\F]$. Let $\yo{\dgs'}\underset{\yo{\dgs}}\times\yo{\dgt}$ be the homotopy pullback computed in $\spr{\Dman[\F]}$. There is a weak equivalence in $\spr{\Dman[\F]}$: 
	\begin{eqn}\label{IntersectionTwoWays}\yo{\dgs'\underset{\dgs}\times\dgt}\overset{\simeq}\longrightarrow
		\yo{\dgs'}\underset{\yo{\dgs}}\times\yo{\dgt}.\n\end{eqn}
The same statement holds for $\Dmana[\F]$ and $\spr{\Dmana[\F]}$.
\end{proposition}
\begin{proof} In the model structure on $\spr{\Dman[\F]}$ the weak equivalences and fibrations are objectwise weak equivalences and fibrations (\cite{HAG1} \S2.3.1), therefore we can compute the homotopy pullback $\yo{\dgs'}\underset{\yo{\dgs}}\times\yo{\dgt}$ also objectwise and use any representatives (up to weak equivalences) of the mapping spaces. Let $\dgt'\in\Dman[\F]$. We have a natural weak equivalence 
		\begin{eqn}\nerve{\cocy\lp\dgt',\dgs\rp}\overset\simeq\hooklongrightarrow
			\mapis[\loca{\Dman[\F]}]\lp\dgt',\dgs\rp,\end{eqn}
where $\nerve{\cocy\lp\dgt',\dgs\rp}$ is the nerve of the category of cocycles from $\dgt'$ to $\dgs$ (proof of Theorem \ref{ManifoldsInclusion}). Let $\cocy^{\dgs'}\lp\dgt',\dgs\rp\subset\cocy\lp\dgt',\dgs\rp$ be the full subcategory consisting of $\dgt'\leftarrow\dgt''\rightarrow\dgs$ with the property that $\hofu[0]\lp\dgt''\rp\rightarrow\hofu[0]\lp\dgs\rp$ factors through $\hofu[0]\lp\dgs'\rp\hookrightarrow\hofu[0]\lp\dgs\rp$. Since morphisms between cocycles are given by weak equivalences, the nerve $\nerve{\cocy^{\dgs'}\lp\dgt',\dgs\rp}$ is a union of connected components of $\nerve{\cocy\lp\dgt',\dgs\rp}$.
\begin{lemma}\label{ConnectedComp} The obvious inclusion $\nerve{\cocy\lp\dgt',\dgs'\rp}\hookrightarrow\nerve{\cocy^{\dgs'}\lp\dgt',\dgs\rp}$, given by composition with $\dgs'\hookrightarrow\dgs$, is a weak equivalence.\end{lemma}
\begin{proof} Composition with $\dgs'\hookrightarrow\dgs$ gives us a functor $\cocy\lp\dgt',\dgs'\rp\rightarrow\cocy^{\dgs'}\lp\dgt',\dgs\rp$. We construct a functor in the opposite direction: given $\dgt'\leftarrow\dgt''\rightarrow\dgs$ in $\cocy^{\dgs'}\lp\dgt',\dgs\rp$ we define $\dgt''':=\dgt''\underset{\dgs}\times\dgs'$. By construction $\dgt'''\rightarrow\dgt''$ is a fibration ($\dgs'\hookrightarrow\dgs$ is a fibration), while the assumption that $\hofu[0]\lp\dgt''\rp\rightarrow\dgs$ factors through $\dgs'$ implies that $\dgt'''\rightarrow\dgt''$ is a weak equivalence. Thus we have an adjunction $\cocy\lp\dgt',\dgs'\rp\leftrightarrows\cocy^{\dgs'}\lp\dgt',\dgs\rp$, hence the corresponding maps on the nerves are weak equivalences.\end{proof}
The previous Lemma tells us that in order to compute the homotopy pullback $\nerve{\cocy\lp\dgt',\dgs'\rp}\underset{\nerve{\cocy\lp\dgt',\dgs\rp}}\times\nerve{\cocy\lp\dgt',\dgt\rp}$ we can equivalently compute the homotopy pullback
	\begin{eqn}\label{HalfComponent}\nerve{\cocy^{\dgs'}\lp\dgt',\dgs\rp}
		\underset{\nerve{\cocy\lp\dgt',\dgs\rp}}\times\nerve{\cocy\lp\dgt',\dgt\rp}.\n\end{eqn}
The map $\nerve{\cocy^{\dgs'}\lp\dgt',\dgs\rp}\rightarrow\nerve{\cocy\lp\dgt',\dgs\rp}$ is an inclusion of a union of connected components. Hence (\ref{HalfComponent}) is naturally equivalent to the corresponding union of connected components in $\nerve{\cocy\lp\dgt',\dgt\rp}$, i.e.\@ to $\nerve{\cocy^{\dgs'\underset{\dgs}\times\dgt}\lp\dgt',\dgt\rp}$. Using the Lemma again we are done. The same proof applies to $\Dmana[\F]$. \end{proof}

According to Proposition \ref{TwoIntersections} Yoneda embeddings on $\Dman[\F]$ and $\Dmana[\F]$ map intersections of open charts to homotopy pullbacks. This allows us to construct useful families of hypercovers in $\spr{\Dman[\F]}$ and $\spr{\Dmana[\F]}$.

\begin{theorem}\label{CechCovers} Let $\dgs\in\Dman[\F]$, and let $\lf U_i\rf_{i\in I}$ be an open affine atlas on $\cs$. For any non-empty finite $s\subseteq I$ let $U_s:=\underset{i\in s}\bigcap\, U_i$, and $\dgs_s:=\lp U_s,\streaf[\bullet]{\dgs}|_{U_s}\rp$. Then
	\begin{eqn}\lf\underset{0\leq m\leq n}\coprod\;\underset{\oset{n}\twoheadrightarrow\oset{m}}\coprod\;
	\yo{\underset{|s|=m+1}\bigsqcup\dgs_s}\rf_{n\geq 0}\longrightarrow\yo{\dgs}\end{eqn}
is a hypercover (\cite{HAG1} Def.\@ 3.2.3).\end{theorem}
\begin{proof} It is clear that $\underset{|s|=1}\bigsqcup\dgs_s\rightarrow\dgs$ is a covering. When $n>0$ Proposition \ref{TwoIntersections} tells us that the map from $\yo{\underset{|s|= n+1}\bigsqcup\dgs_s}$ to the fiber product of copies of $\yo{\underset{|s|= n}\bigsqcup\dgs_s}$ over $\yo{\dgs}$ is a weak equivalence, hence a covering.\end{proof}

Now we turn to the categories of stacks on $\op{\derca[\F]}$, $\Dmana[\F]$ and $\Dman[\F]$. We will denote them by $\stacks{\op{\derca[\F]}}$, $\stacks{\Dmana[\F]}$ and $\stacks{\Dman[\F]}$ respectively. These are model categories obtained as left Bousfield localizations (at the hypercovers or equivalently local weak equivalences) of $\spr{\loca{\op{\derca[\F]}}}$, $\spr{\loca{\Dmana[\F]}}$ and $\spr{\loca{\Dman[\F]}}$ (\cite{HAG1} Rem.\@ 3.4.6). We would like to compare these 3 categories of stacks. 

\begin{lemma}\label{PreservingLocal} In the Quillen adjunction $L\colon\spr{\Dmana[\F]}\leftrightarrows\spr{\Dman[\F]}\colon R$ the left Quillen functor $L$ maps local weak equivalences to local weak equivalences.\end{lemma}
\begin{proof} Let $\eta\colon F\rightarrow G$ be a local weak equivalence in $\spr{\Dmana[\F]}$. According to Lemma 3.3.3 in \cite{HAG1} this is equivalent to $\eta$ being a hypercover (between constant simplicial objects in $\spr{\Dmana[\F]}$). We would like to show that $L(\eta)$ is also a hypercover (in $\spr{\Dman[\F]}$). First we notice that every object in $\Ho{\Dman[\F]}$ has a covering consisting of objects in $\Ho{\Dmana[\F]}$. This implies that $\mu\colon F'\rightarrow G'$ in $\spr{\Dman[\F]}$ is a covering (\cite{HAG1} Def.\@ 3.1.3), if and only if $R(\mu)$ is a covering. Then $\forall G'\in\spr{\Dman[\F]}$ the counit $L(R(G'))\rightarrow G'$ is a covering. For $L(\eta)$ to be a hypercover the following morphisms need to be coverings (\cite{HAG1} Lemma 3.3.3)
	\begin{eqn}\forall n\geq 0\quad L(F)^{\Delta_n}\longrightarrow
	L(F)^{\partial\Delta_n}\underset{L(G)^{\partial\Delta_n}}{\times^h} L(G)^{\Delta_n}.\end{eqn}
Using adjunction and the fact that the counit consists of coverings we have that these conditions are equivalent to
	\begin{eqn}\forall n\geq 0\quad F^{\Delta_n}\longrightarrow
	RL(F)^{\partial\Delta_n}\underset{RL(G)^{\partial\Delta_n}}{\times^h} RL(G)^{\Delta_n}\simeq
	F^{\partial\Delta_n}\underset{G^{\partial\Delta_n}}{\times^h} G^{\Delta_n}\end{eqn}
being coverings.\end{proof}

\begin{theorem}\label{QuillenEq} The Quillen adjunction $\spr{\Dmana[\F]}\leftrightarrows\spr{\Dman[\F]}$ defines a Quillen equivalence $\stacks{\Dmana[\F]}\leftrightarrows\stacks{\Dman[\F]}$.\end{theorem}
\begin{proof} Since cofibrations in $\stacks{\Dmana[\F]}$ are the same as in $\spr{\Dmana[\F]}$ and similarly for $\Dman[\F]$, it is clear that the left adjoint preserves cofibrations. Lemma \ref{PreservingLocal} tells us that this left adjoint also preserves local weak equivalence, hence it is a left Quillen functor.
	
According to Theorem \ref{CechCovers} $\forall\dgs\in\Dman[\F]$ has a hypercover consisting of objects in $\Dmana[\F]$, hence every stack on $\Dman[\F]$ is determined (up to an isomorphism in $\Ho{\stacks{\Dman[\F]}}$) by its values on objects in $\Dmana[\F]$. This implies that the counit of $\Ho{\stacks{\Dmana[\F]}}\leftrightarrows\Ho{\stacks{\Dman[\F]}}$ consists of isomorphisms.\end{proof}

\begin{theorem} The Quillen adjunction $\spr{\Dmana[\F]}\leftrightarrows\spr{\op{\derca[\F]}}$ defines a Quillen adjunction $\stacks{\Dmana[\F]}\leftrightarrows\stacks{\op{\derca[\F]}}$.\end{theorem}
\begin{proof} First we state some obvious properties of the étale topologies we are using.\begin{enumerate}
	\item The topologies on $\op{\derca}$, $\Dmana[\F]$, $\op{\fderca[\F]}$ are quasi-compact, i.e.\@ they are generated by finite covering families.\footnote{Recall that $\fderca[\F]\subset\derca[\F]$ is the essential image of $\sderca[\F]$.}
	\item For each finite family (possibly empty) $\lf \dgs_i\rf_{i\in I}$ of objects the family of inclusions $\lf\dgs_i\hookrightarrow\underset{j\in I}\bigsqcup\;\dgs_j\rf_{i\in I}$ is a covering.\end{enumerate}
These are the first two parts of Assumption 1.3.2.2 in \cite{HAG2}. Using these properties we have that Lemma 1.3.2.3 from \cite{HAG2} applies to $\stacks{\op{\fderca[\F]}}$ and $\stacks{\op{\derca[\F]}}$
\begin{lemma}\begin{enumerate}
	\item Let $\lf\dgs_i\rf_{i\in I}$ be a finite family of objects in either $\op{\derca[\F]}$ or $\op{\fderca[\F]}$. The natural morphism
		\begin{eqn}\label{Coproducts}\underset{i\in I}\coprod\;\yo{\dgs_i}\longrightarrow
			\yo{\underset{i\in I}\bigsqcup\;\dgs_i}\n\end{eqn}
	is a weak equivalence in the corresponding category of stacks.
	\item The model categories of stacks on $\op{\derca[\F]}$, $\op{\fderca[\F]}$ are left Bousfield localizations of the corresponding categories of prestacks with respect to representable hypercovers and morphisms as in (\ref{Coproducts}).\end{enumerate}\end{lemma}
\begin{proof} In the case of $\op{\derca[\F]}$ we can use Lemma 1.3.2.3 from \cite{HAG2} directly. In the case of $\op{\fderca[\F]}$ we have a pseudo-model category (\cite{HAG1} \S4) with $\op{\derca[\F]}$ being the ambient model category. Hence the proof of Lemma 1.3.2.3 from \cite{HAG2} applies here as well.\end{proof}
Using the Lemma above we can (as in \cite{HAG2} \S2.2.4) obtain that the left adjoint in $\stacks{\op{\fderca[\F]}}\leftrightarrows\stacks{\op{\derca[\F]}}$ is a left Quillen functor also with respect to the local model structure by noticing that $\op{\fderca[\F]}\hookrightarrow\op{\derca[\F]}$ preserves coproducts, weak equivalences and hypercovers. Composing with the Quillen equivalence $\stacks{\Dmana[\F]}\leftrightarrows\stacks{\op{\fderca[\F]}}$ we are done.\end{proof}

We have shown that the category $\stacks{\Dman[\F]}$ is (weakly) equivalent to the category $\stacks{\op{\fderca[\F]}}$, which is obtained by restricting stacks defined on $\op{\derca[\F]}$ to the essential image $\op{\fderca}$ of $\op{\sderca[\F]}$ in $\op{\derca[\F]}$. It is obvious that $\op{\fderca[\F]}$ is a pseudo-model category with $\op{\derca[\F]}$ being the ambient model category (\cite{HAG1} \S4). Hence $\stacks{\op{\fderca[\F]}}$ inherits from $\stacks{\op{\derca[\F]}}$ the definitions of derived schemes and derived $n$-stacks with respect to smooth morphisms (e.g.\@ \cite{To10} \S5.2).

\begin{theorem}\label{ManSche} Let $\dgs\in\Dman[\F]$, and let $\yosh{\dgs}$ be the corresponding stack on $\op{\sderca[\F]}$, then $\yosh{\dgs}$ is a derived scheme.\end{theorem}
\begin{proof} Let $\lf U_i\rf_{i\in I}$ be a finite affine open atlas on $\cs$, and let $\lf\dgs_i\rf_{i\in I}$ be the corresponding dg affine atlas. Theorem \ref{CechCovers} tells us that $\lf\yo{\dgs_i}\rf_{i\in I}$ define a hypercover of $\yo{\dgs}$, in particular $\underset{i\in I}\coprod\;\yo{\dgs_i}\rightarrow\yosh{\dgs}$ is an epimorphism. Since the associated stack functor preserves homotopy fiber products (\cite{HAG1} Prop.\@ 3.4.10), Proposition \ref{TwoIntersections} tells us that each $\yo{\dgs_i}\rightarrow\yosh{\dgs}$ is a Zariski open immersion.\end{proof}


\begin{thebibliography}{Bour}{\small
\bibitem{Borelli67} M.Borelli. {\it Some results on ampleness and divisorial schemes.} Pacific Journal of Mathematics, vol.\@ 23 no.\@ 2, pp.\@ 217-227 (1967).
\bibitem{BSY} D.Borisov, A.Sheshmani, S.-T.Yau. {\it Global shifted potentials on moduli spaces of sheaves on Calabi--Yau $4$-folds.} In preparation.
\bibitem{BSY2} D.Borisov, A.Sheshmani, S.-T.Yau. {\it Global shifted potentials for moduli stacks of sheaves on Calabi-Yau four-folds II (the stable locus)} arXiv:2007.13194 [math.AG].
\bibitem{BKan} A.K.Bousfield, D.M.Kan. {\it Homotopy limits, Completions and Localizations.} Lecture Notes in Mathematics 304, Springer (1972).
\bibitem{Brown73} K.S.Brown. {\it Abstract homotopy theory and generalized sheaf cohomology.} Transactions of the AMS vol.\@ 186, pp.\@ 419-458 (1973).
\bibitem{DerQuot} I.Ciocan-Fontanine, M.Kapranov. {\it Derived Quot schemes.} Ann.\@ Scient.\@ \'Ec/\@ Norm.\@ Sup.\@ 4th series, t.\@ 34 p.\@ 403-440 (2001).
\bibitem{DHKS} W.G.Dwyer, Ph.S.Hirshhorn, D.M.Kan, J.H.Smith. {\it Homotopy Limit Functors on Model Categories and Homotopical Categories.} AMS (2004).
\bibitem{DK1} W.G.Dwyer, D.M.Kan. {\it Simplicial localizations of categories.} J.\@ of Pure and Applied Algebra 17, p.\@ 267-284 (1980).
\bibitem{DK2} W.G.Dwyer, D.K.Kan. {\it Calculating simplicial localizations.} J.\@ of Pure and Applied Algebra 18, p.\@ 17-35 (1980).
\bibitem{DK3} W.G.Dwyer, D.K.Kan. {\it Function complexes in homotopical algebra.} Topology vol.\@ 19 p.\@ 427-440 (1980).
\bibitem{GaiRoz17} D.Gaitsgory, N.Rozenblyum. {\it A Study in Derived Algebraic Geometry. Volume II: Deformations, Lie Theory and Formal Geometry.} AMS (2017).
\bibitem{Gambino} N.Gambino. {\it Weighted limits in simplicial homotopy theory.} Journal of Pure and Applied Algebra 214, pp.\@ 1193-1199 (2010).
\bibitem{Hartshorne} R.Hartshorne. {\it Algebraic geometry.} Springer (1977).
\bibitem{Hirsh} Ph.S.Hirshhorn. {\it Model categories and their localizations.} AMS (2003).
\bibitem{Hovey} M.Hovey. {\it Model categories.} AMS (1999).
\bibitem{HuyLehn} D.Huybrechts, M.Lehn. {\it The geometry of moduli spaces of sheaves.} Cambridge University Press (2010).
\bibitem{Low} Zh.L.Low. {\it Cocycles in categories of fibrant objects.} arXiv:1502.03925 [math.CT]
\bibitem{Ma03} M.A.Mandell. {\it Topological Andr\'e--Quillen cohomology and $E_\infty$ Andr\'e--Quillen cohomology.} Adv. Math. 177 no. 2, pp. 227-279, (2003).
\bibitem{PTVV13} T.Pantev, B.To\"en, M.Vaqui\'e, G.Vezzosi. {\it Shifted symplectic structures.} Publ. Math. Inst. Hautes \'Etudes Sci. 117, pp. 271-328 (2013).
\bibitem{Qu67} D.Quillen. {\it Homotopical algebra.} Lecture Notes in Mathematics 43, Springer (1967).
\bibitem{DoldKan} S.Schwede, B.Shipley. {\it Equivalences of monoidal model categories.} Algebraic \& Geometric Topology vol.\@ 3, p.\@ 287-334 (2003).
\bibitem{To14} B.To\"en. {\it Derived algebraic geometry.} EMS Surv. Math. Sci. 1, no. 2, pp. 153-240 (2014).
\bibitem{ToVa} B.To\"en, M.Vaqui\'e. {\it Moduli of objects in dg-categories.} Ann.\@ Scient.\@ \'Ec.\@ Norm.\@ Sup.\@ 4th series, t.\@ 40, p.\@ 387-444 (2007). 
\bibitem{HAG1} B.Toën, G.Vezzosi. {\it Homotopical algebraic geometry I: topos theory.} Advances in Mathematics 193, pp.\@ 257-372 (2005).
\bibitem{HAG2} B.Toën, G.Vezzosi. {\it Homotopical algebraic geometry II: Geometric stacks and applications.} Mem. Amer. Math. Soc. 193, no. 902 (2008).
\bibitem{To10} B. To\"{e}n. {\it Simplicial presheaves and derived algebraic geometry.} pp. 119-186 in I.Moerdijk, B.To\"{e}n. {\it Simplicial methods for operads and algebraic geometry.} Birkh\"{a}user (2010).

}


\end{thebibliography}
\end{document}